\documentclass[11pt,a4paper]{article}
\usepackage[toc,page]{appendix}

\usepackage[all,cmtip]{xy}
\entrymodifiers={+!!<0pt,\fontdimen22\textfont2>}

\usepackage{fouriernc}
\usepackage[T1]{fontenc}
\usepackage[english]{babel}         %outputs according to rules of language
\usepackage{verbatim}               %outputs text verbatim via \begin{verbatim}\end{verbatim}

\usepackage{amsmath}
\usepackage{amsfonts}
\usepackage{amssymb}
\usepackage{mathrsfs}    %allows \mathscr
\usepackage{amsthm}
\usepackage{xcolor}

\usepackage{hyperref}           %Hyperlinks

\usepackage{cite}
\usepackage{stmaryrd}

\usepackage{graphicx}
\usepackage{pgf}
\usepackage{eepic}
\usepackage{pstricks}
\usepackage[all,cmtip]{xy}

\newcommand{\spa}[1]{\mathrm{Spa}\left(#1\right)}
\newcommand{\spf}[1]{\mathrm{Spf}\left(#1\right)}

\newcommand{\et}{\mathrm{\'{e}t}}

\usepackage{epsfig}   % for xfig

\theoremstyle{plain}
\newtheorem*{thrm}{Theorem}
\newtheorem{theo}{Theorem}[section]

\newtheorem{prop}[theo]{Proposition}
\newtheorem{cor}[theo]{Corollary}

\newtheorem{lem}[theo]{Lemma}

\newtheorem{guess[theo]}{Guess}
\theoremstyle{definition}

\newtheorem{defn}[theo]{Definition}

\theoremstyle{remark}
\newtheorem{rem}[theo]{Remark}

\newcommand{\proofof}[1]{\end{#1}\begin{proof}}

\makeatletter
\renewcommand\section{\@startsection {section}{1}{\z@}%
  {-3.5ex \@plus -1ex \@minus -.2ex}{2.3ex \@plus.2ex}%
  {\normalfont\large\bfseries}}
\renewcommand\subsection{\@startsection{subsection}{2}{\z@}%
  {-3.25ex\@plus -1ex \@minus -.2ex}{1.5ex \@plus .2ex}%
  {\normalfont\bfseries}}

\newcommand{\sh}[1]{\mathcal{#1}}

\newcommand{\Z}{{\mathbb Z}}
\newcommand{\Q}{{\mathbb Q}}

\newcommand{\A}{{\mathbb A}}
\newcommand{\D}{{\mathbb D}}

\DeclareMathAlphabet{\mathrmsl}{OT1}{cmr}{m}{sl}

\newcommand{\rssymb}[2]{\newcommand{#1}{\mathrmsl{#2}} }
\newcommand{\oper}[3][n]{\newcommand{#2}{\mathop{\mathrm{#3}}%
\ifx n#1\nolimits\else\limits\fi} }
\newcommand{\rsoper}[3][n]{\newcommand{#2}{\mathop{\mathrmsl{#3}}%
\ifx n#1\nolimits\else\limits\fi} }

\newcommand{\lser}[1]{(\!(#1)\!)}
\newcommand{\pow}[1]{\llbracket #1 \rrbracket}

\newcommand{\spec}[1]{\mathrm{Spec}\left(#1\right)}
\newcommand{\cur}[1]{\mathcal{#1}}

\newcommand{\norm}[1]{\left\vert#1\right\vert}
\newcommand{\isomto}{\overset{\sim}{\rightarrow}}
\newcommand{\bu}{\bullet}

\newcommand{\rig}{\mathrm{rig}}
\newcommand{\ek}{\cur{E}_K}
\newcommand{\ekd}{\cur{E}_K^\dagger}
\newcommand{\tate}[1]{\langle #1 \rangle}

\newcommand{\pn}{(\varphi,\nabla)}
\newcommand{\rk}{\cur{R}_K}
\oper\Ad{Ad}

\oper\val{val}
\oper\coker{coker}
\oper\mult{mult}
\oper\Iso{Iso}
\oper\End{End}
\oper\Aut{Aut}
\oper\Sub{Sub}
\oper\Alt{Alt}
\oper\Ext{Ext}
\oper\Pic {Pic}
\oper\Sym{Sym}
\oper\Spec{Spec}
\oper\Spf{Spf}
\oper\Sp{Sp}
\oper\Spa{Spa}
\oper\Proj{Proj}
\rsoper\divg{div}
\rsoper{\sym}{sym}
\rsoper{\alt}{alt}
\rsoper\trace{tr}
\rssymb\id{id}

\newcommand{\thismonth}{\ifcase\month\or
  January\or February\or March\or April\or May\or June\or
  July\or August\or September\or October\or November\or December\fi
  \space\number\year}

\title{Rigid cohomology over Laurent series fields III: Absolute coefficients and arithmetic applications}
\author{Christopher Lazda and Ambrus P\'{a}l}

\begin{document}

\maketitle 

\abstract{In this paper we investigate the arithmetic aspects of the theory of $\ekd$-valued rigid cohomology introduced and studied in \cite{rclsf1,rclsf2}. In particular we show that these cohomology groups have compatible connections and Frobenius structures, and therefore are naturally $\pn$-modules over $\ekd$ whenever they are finite dimensional. We also introduce a category of `absolute' coefficients for the theory; the same results are true for cohomology groups with coefficients. We moreover prove a $p$-adic version of the weight monodromy conjecture for smooth (not necessarily proper) curves, and use a construction of Marmora to prove a version of $\ell$-independence for smooth curves over $k\lser{t}$ that includes the case $\ell=p$. This states that after tensoring with $\cur{R}_K$, our $p$-adic cohomology groups agree with the $\ell$-adic Galois representations $H^i_\et(X_{k\lser{t}^\mathrm{sep}},\Q_\ell)$ for $\ell\neq p$.}

\tableofcontents

\section*{Introduction}\addcontentsline{toc}{section}{Introduction}

This is the final paper in the series \cite{rclsf1,rclsf2} concerning a new $p$-adic cohomology theory for varieties over Laurent series fields in positive characteristic. A general overview of the whole series is given in the introduction to \cite{rclsf1}, so here we will summarise the results of the previous papers and give a detailed introduction to the results contained here.

Let $k$ be a field of characteristic $p>0$. In the first paper \cite{rclsf1} we introduced a version of rigid cohomology
$$ X\mapsto H^*_\rig(X/\ekd)
$$ for varieties over the Laurent series field $k\lser{t}$ with values in vector spaces over the bounded Robba ring $\ekd$ (here $K$ is a complete discretely valued field of characteristic $0$ with residue field $k$). There we proved that the cohomology groups were well-defined and functorial, as well as introducing categories of coefficients. In the second paper \cite{rclsf2} we proved finite dimensionality and base change for smooth curves, as well as introducing compactly supported cohomology and proving a version of Poincar\'{e} duality. 

These results can be seen in some sense to describe the `geometric' properties of the theory, and in this paper we concentrate on exploring the corresponding `arithmetic' properties, that is the extra structures on these cohomology groups which will reflect properties such as the reduction type of the variety under consideration. The first part of the paper is therefore dedicated to endowing the cohomology groups $H^*_\rig(X/\ekd)$ with a canonical connection, the Gauss--Manin connection, and to this end we introduce a new category of `absolute' coefficients $\mathrm{Isoc}^\dagger(X/K)$ associated to any variety $X/k\lser{t}$ in which the objects have differential structures relative to $K$ rather than $\ekd$. The key observation in allowing us to set up this theory is the fact that when the base field $k$ has a finite $p$-basis, the natural inclusion
$$ k\pow{t}\cdot  dt \rightarrow \Omega^1_{k\pow{t}/k}
$$
is actually an isomorphism. By boot-strapping up, this enable us to do differential calculus relative to $\cur{V}$ on $p$-adic formal schemes over $\cur{V}\pow{t}$, with essentially no modifications required to the usual procedure. This can then be easily carried over to rigid varieties over $S_K$ in the sense of \cite{rclsf1}. The upshot of all this is then the fact that we can make exactly the same definitions and constructions to build the category $\mathrm{Isoc}^\dagger(X/K)$ as we did to build the category $\mathrm{Isoc}^\dagger(X/\ekd)$ in \S5 of \emph{loc. cit.}, provided we work with the category of $p$-adic formal schemes of `pseudo-finite type' over $\cur{V}$, i.e. of finite type over some $\spf{\cur{V}\pow{t}} \times_\cur{V} \ldots \times_\cur{V} \spf{\cur{V}\pow{t}}$. There will be a canonical forgetful functor 
$$ \mathrm{Isoc}^\dagger(X/K)\rightarrow \mathrm{Isoc}^\dagger(X/\ekd)
$$
and since objects in the former category have differential structure relative to $K$, the usual construction will give a Gauss--Manin connection on their cohomology. When $\cur{E}\in F\text{-}\mathrm{Isoc}^\dagger(X/\ekd)$, then simple functoriality of cohomology will give rise to a Frobenius on $H^*_\rig(X/\ekd,\cur{E})$, and this will be compatible with the Gauss--Manin connection when $\cur{E}\in F\text{-}\mathrm{Isoc}^\dagger(X/\ekd)$. Hence we get the following.

\begin{thrm}[\ref{pnmodst}] Let $X/k\lser{t}$ be a variety, and assume that $k$ has a finite $p$-basis. Let $\cur{E}\in F\text{-}\mathrm{Isoc}^\dagger(X/K)$ and $i\geq0$. Assume that the base change morphism
$$ H^i_\rig(X/\ekd,\cur{E})\otimes_{\ekd}\ek \rightarrow H^i_\rig(X/\ek,\hat{\cur{E}})
$$
is an isomorphism. Then $H^i_\rig(X/\ekd,\cur{E})$ comes with a canonical structure as a $\pn$-module over $\ekd$.
\end{thrm}

This is the case in particular when $X$ is a smooth curve over $k\lser{t}$, in which case it is also straightforward to verify that the Poincar\'{e} pairing is a perfect pairing of $\pn$-modules.

The focus of the second section is then in exploring the expected connections between the arithmetic properties of some variety $X/k\lser{t}$ and the $\pn$-module structure on its cohomology $H^*_\rig(X/\ekd)$, or rather the associated $\pn$-module
$$ H^i_\rig(X/\cur{R}_K):= H^i_\rig(X/\ekd)\otimes_{\ekd} \rk
$$
over the Robba ring $\rk$. The point is that for any variety $X/k\lser{t}$ to which we can apply the above theorem, in particular for smooth curves, we can attach $p$-adic Weil--Deligne representations to their cohomology groups by first base changing to $\rk$ and then using Marmora's functor
$$ \mathrm{WD}: \underline{\mathrm{M}\Phi}^\nabla_{\cur{R}_K} \rightarrow \mathrm{Rep}_{K^\mathrm{un}}(\mathrm{WD}_{k\lser{t}}).
$$
from $\pn$-modules over $\rk$ to Weil--Deligne representations with values in the maximal unramified extension of $K$. This will be our main tool in the study of the arithmetic properties of $p$-adic cohomology of such varieties, and the two main questions we are interested in are the $p$-adic weight-monodromy conjecture and $\ell$-independence. Specifically, the main results we prove concerning smooth curves are the following.

\begin{thrm}[\ref{wmcp}, \ref{lindepopen}] Let $k$ be a finite field, and suppose that $K=W(k)[1/p]$ is the fraction field of the Witt vectors of $K$. Let $X/k\lser{t}$ be a smooth curve, let $\overline{X}$ be a smooth compactification of $X$ and assume that $D=\overline{X}\setminus X$ is geometrically reduced. \begin{enumerate}
\item For all $i\geq0$ there is a natural filtration $^gW_\bu$ on $H^i_\rig(X/\rk)$, the geometric weight filtration, such that for all $n$, the Weil--Deligne representation attached to the $n$th graded piece $$\mathrm{Gr}_n^{^gW}(H^i_\rig(X/\cur{R}_K))$$ is `quasi-pure' of weight $n$, that is the $k$th graded piece of the monodromy filtration is pure of weight $n+k$. Moreover, if $X$ is proper (i.e. $D=\emptyset$) then the only non-zero graded piece is $\mathrm{Gr}_i^{^gW}$. 
\item The Weil--Deligne representation attached to $H^i_\rig(X/\rk)$ is `compatible' with the family of Weil--Deligne representations attached to the Galois representations $$H^i_\et(X_{k\lser{t}^\mathrm{sep}},\Q_\ell)$$ for $\ell\neq p$.
\end{enumerate}
\end{thrm}

Although we only know base change for smooth curves, we can actually prove cases of $\ell$-independence and the weight-monodromy conjecture for any smooth and proper variety over $k\lser{t}$ using Kedlaya's overconvergence theorem for the cohomology of such varieties (Theorem 7.0.1 of \cite{kedthesis}). We can also prove a $p$-adic criterion for good/semistable reduction of abelian varieties analogous to the N\'{e}ron--Ogg--Shafarevich criterion. These results are actually little more than a rephrasing of those from \cite{sga719} on $p$-divisible groups of abelian varieties in terms of $\pn$-modules over $\rk$, using theorems of Kedlaya and de Jong comparing $p$-divisible groups and Dieudonn\'{e} modules to be able to reinterpret Grothendieck's results in the expected manner. The main results for smooth and proper varieties are almost certainly known to the experts and do not in fact use our construction of $H^i_\rig(X/\ekd)$ in any way. However, there are several reasons for including them here, and in particular including them as an application of the theory of $\ekd$-valued rigid cohomology. Firstly, the precise description of the weight and monodromy filtrations on $H^i_\rig(X/\cur{R}_K)$ in the smooth and proper case, and the cohomological criterion for good or semistable reduction of abelian varieties that we present are far more transparently analogous to the corresponding statements in the $\ell$-adic case than those given in \cite{sga719}, even though the results in \emph{loc. cit.} provide the majority of the actual content of our own. Secondly, the fact that we have a cohomology theory for singular or open varieties, not just smooth and proper ones, allows us to easily deduce versions of weight-monodromy and $\ell$-independence for open curves (as stated above) from those from complete curves. We expect that once we know base change in general, versions of these results for singular or open varieties can be deduced (although not in a completely straightforward manner) from the corresponding ones for smooth and proper varieties. Finally, and perhaps most importantly, a full treatment of these questions will necessarily involve a $p$-adic vanishing cycles formalism, and the construction of a $p$-adic weight spectral sequence. Perhaps the most natural (or at least the most obvious to us) way of obtaining such a formalism will be via a suitably robust theory of arithmetic $\cur{D}$-modules on varieties over $k\lser{t}$ and/or $k\pow{t}$, and our construction of $H^i_\rig(X/\ekd)$ is perhaps the framework for $p$-adic cohomology which most easily leads into the constriction of such a theory. We hope to be able to explore this in greater depth in future work.

\section{Absolute coefficients and the Gauss--Manin connection} \label{absgm}

The set up and notation will be almost exactly as in \cite{rclsf1,rclsf2}. We will denote by $k$ a field of characteristic $p>0$, $K$ a complete, discretely valued field of characteristic $0$ with residue field $k$ and ring of integers $\cur{V}$, $\pi$ a uniformiser for $K$ and $k\lser{t}$ the ring of Laurent series over $k$. However, since we will want to make use of the simple characterisation of $\Omega^1_{k\pow{t}}$ below, we will assume further that $k$ has a finite $p$-basis. We will denote by $\ekd$ the bounded Robba ring over $K$, $\cur{R}_K$ the Robba ring and $\ek$ the Amice ring (for definitions of these see \emph{loc. cit.}). We will fix compatible Frobenii $\sigma$ on the rings $\cur{V},K,\cur{V}\pow{t},S_K:=\cur{V}\pow{t}\otimes_\cur{V} K, \ekd,\ek$ and $\cur{R}_K$. 

In \cite{rclsf1} we constructed a version of rigid cohomology
$$ X\mapsto H^*_\rig(X/\ekd)
$$
for varieties (i.e. separated schemes of finite type) over $k\lser{t}$, taking values in $\ekd$-vector spaces, and in \cite{rclsf2} prove that for smooth curves the natural base change morphism
$$ H^*_\rig(X/\ekd)\otimes_{\ekd}\ek \rightarrow  H^*_\rig(X/\ekd)
$$
to classical rigid cohomology was an isomorphism (actually, this result also holds with coefficients). This therefore gives us a satisfactory \emph{geometric} theory, in that $H^*_\rig(X/\ekd)$ is finite dimensional and of the expected dimension. As explained in the introduction, we expect these groups to come with the addition structure of $(\varphi,\nabla)$-modules over $\ekd$, which will reflect the particular \emph{arithmetic} properties of the situation under consideration.

\begin{defn}  \label{modules} Let $\partial_t:\ekd\rightarrow \ekd$ be the $K$-derivation given by differentiation with respect to $t$.
\begin{itemize}
\item A $\varphi$-module over $\ekd$ is a finite dimensional $\ekd$-vector space $M$ together with a Frobenius structure, that is an $\sigma$-linear map $$\varphi: M \rightarrow M$$which induces an isomorphism $M\otimes_{\ekd,\sigma} \ekd\cong M$.
\item A $\nabla$-module over $\ekd$ is a finite dimensional $\ekd$-vector space $M$ together with a connection, that is an $K$-linear map $$\nabla:M\rightarrow  M$$ such that $\nabla(fm)=\partial_t(f)m+f\nabla(m)$ for all $f\in\ekd$ and $m\in M$.
\item A $(\varphi,\nabla)$-module over $\ekd$ is a finite dimensional $\ekd$-vector space $M$ together with a Frobenius $\varphi$ and a connection $\nabla$, such that the diagram
$$
\xymatrix{
M \ar[r]^\nabla \ar[d]^{\varphi} & M \ar[d]^{\partial_t(\sigma(t))\varphi} \\
M \ar[r]^\nabla & M
}
$$
commutes.
\end{itemize}
\end{defn}

It is this extra structure that will reflect the arithmetic properties of the situation, for example if $A/k\lser{t}$ is an abelian variety, then we expect to be able to `read off' whether or not $A$ has good reduction from the $(\varphi,\nabla)$-module structure on its cohomology $H^*_\rig(A/\ekd)$.

The origin of the expected $\varphi$-module structure on $H^*_\rig(X/\ekd)$ is entirely straightforward: for any variety $X/k\lser{t}$, and any $\sh{E}\in F\text{-}\mathrm{Isoc}^\dagger(X/\ekd)$, functoriality gives a $\sigma$-linear endomorphism
$$
\varphi:H^*_\rig(X/\ekd,\sh{E}) \rightarrow H^*_\rig(X/\ekd,\sh{E})
$$
corresponding to a linear map
$$
\varphi^\sigma: H^*_\rig(X/\ekd,\sh{E})\otimes_{\ekd,\sigma} \ekd \rightarrow H^*_\rig(X/\ekd,\sh{E}).
$$
Of course, we expect this latter map to be an isomorphism for all $X$, but we cannot currently prove this.

\begin{lem} \label{bijfrob} If $X$ is a smooth curve over $k\lser{t}$, then $\varphi^\sigma$ is an isomorphism.
\end{lem}

\begin{proof} This follows from base change - we have compatible Frobenii on $\ekd$ and $\ek$, and $\varphi^\sigma\otimes_{\ekd} \ek$ is the usual Frobenius structure on rigid cohomology $H^*_\rig(X/\ek,\hat{\sh{E}})$. This is an isomorphism, and hence $\varphi^\sigma$ is an isomorphism.
\end{proof}

Thus for any overconvergent $F$-isocrystal $\sh{E}$ on $X/\ekd$, with $X/k\lser{t}$ a smooth curve, the cohomology groups $H^*_\rig(X/\ekd/\sh{E})$ have the structure of $\varphi$-modules over $\ekd$. The origin of the $\nabla$-structure, while no less mysterious, is somewhat more involved, and is the subject of the rest of this section. The basic point is to equip our isocrystals with connections relative to $K$, rather than to $S_K$, and the connection on their cohomology with then simply be the Gauss--Manin connection. To implement this strategy, however, we must first discuss differential calculus relative to $K$ on rigid varieties over $S_K$.

\subsection{Differential calculus on rigid varieties over $S_K$}

In this section, we develop the rudiments of differential calculus on rigid varieties over $S_K$, with differential structure relative to $K$. This will then be used to construct categories of isocrystals with naturally occurring Gauss--Manin connections on their cohomology. The starting point is the following crucial result due to de Jong.

\begin{lem} Let $\Omega^1_{k\pow{t}/k}$ denote the module of differentials of $k\pow{t}/k$, and $\widehat{\Omega}^1_{k\pow{t}/k}$ the module of $t$-adically continuous differentials. Then the natural map
$$
\Omega^1_{k\pow{t}/k} \rightarrow  \widehat{\Omega}^1_{k\pow{t}/k} \cong k\pow{t} \cdot dt
$$
is an isomorphism of $k\pow{t}$-modules. If $A_0$ is a finitely generated $k\pow{t}$-algebra then there is an exact sequences of $A_0$-modules
$$
A_0 \cdot dt  \rightarrow \Omega^1_{A_0/k} \rightarrow \Omega^1_{A_0/k\pow{t}}\rightarrow 0
$$
and in particular $\Omega^1_{A_0/k}$ is a finitely generated $A_0$-module. Moreover, if $A_0$ is smooth over $k\pow{t}$ then this sequence is exact on the left as well, and $\Omega^1_{A_0/k}$ is a projective $A$-module.
\end{lem}

\begin{proof} The first claim follows from Lemma 1.3.3 of \cite{dej1}, and the second then simply follows from the exact sequence
$$
 A\otimes \Omega^1_{k\pow{t}/k}  \rightarrow \Omega^1_{A_0/k} \rightarrow \Omega^1_{A_0/k\pow{t}}\rightarrow 0
$$
associated to $k\rightarrow k\pow{t}\rightarrow A$. 
\end{proof}

Write $\cur{V}_n=\cur{V}/\pi^{n+1}$. For a topologically finitely generated $\cur{V}\pow{t}$-algebra $A$, with reductions $A_n= A/\pi^{n+1}$, we let
$$
\Omega^1_{A/\cur{V}} := \varprojlim_ n \Omega^1_{A_n/\cur{V}_n}
$$
be the module of $\pi$-adically continuous differentials. By the previous lemma, this is a finitely generated projective $A$-module, and there is an exact sequence
$$
 A\cdot dt \rightarrow \Omega^1_{A/\cur{V}} \rightarrow \Omega^1_{A/\cur{V}\pow{t}}\rightarrow 0
$$
of $A$-modules, where 
$$
 \Omega^1_{A/\cur{V}\pow{t}} := \varprojlim_n \Omega^1_{A_n/\cur{V}_n\pow{t}}
$$
is the module of $\pi$-adically continuous differentials on $A/\cur{V}\pow{t}$. Moreover, if $A$ is formally smooth over $\cur{V}\pow{t}$, then this sequence is exact on the left as well. This globalises, and hence for any formal scheme $\frak{X}$ over $\cur{V}\pow{t}$ there is an exact sequence
$$
 \cur{O}_\mathfrak{X} \cdot dt \rightarrow \Omega^1_{\mathfrak{X}/\cur{V}} \rightarrow \Omega^1_{\mathfrak{X}/\cur{V}\pow{t}}\rightarrow 0
$$
which is exact on the left when $\mathfrak{X}$ is smooth over $\cur{V}\pow{t}$. To take this further, we will need to compare modules of differentials and smoothness of a formal scheme over $\cur{V}\pow{t}$ to that of its generic fibre $\mathfrak{X}_K$. The latter is constructed by Huber in \cite{huber2}, for separated morphisms $f:X\rightarrow Y$ locally of finite type, and by definition is it $\Delta^*(\cur{I}) =\cur{I}/\cur{I}^2$ where $\Delta:X\rightarrow X\times_Y X$ is the diagonal morphism, and $\cur{I}\subset \cur{O}_{X\times_YX}$ is the ideal of the diagonal.

\begin{lem} \label{smoothmodel} Let $f:\mathfrak{X}\rightarrow \mathfrak{Y}$ be a morphism of formal schemes of finite type over $\cur{V}\pow{t}$, with generic fibre $f_K:\sh{X}\rightarrow \sh{Y}$. Then there is an isomorphism
$$
\Omega^1_{\sh{X}/\sh{Y}} \cong \Omega^1_{\mathfrak{X}/\mathfrak{Y}} \otimes_{\cur{O}_{\mathfrak{X}}} \cur{O}_{\sh{X}}
$$
which respects the derivations of $\cur{O}_{\sh{X}}$ into both.
\end{lem}

\begin{proof} This just follows from the explicit construction of both modules, i.e. \S1.6.2 of \cite{huber2} and Corollary I.5.1.13 of \cite{rigidspaces}.
\end{proof}

We would like to extend this to define the module of differentials $\Omega^1_{\sh{X}/K}$ for a rigid variety over $S_K$, that is a rigid space separated and locally of finite type over the `bounded open unit disc'
$$\D^b_K:=\spa{S_K,\cur{V}\pow{t}},$$
and compare it with $\Omega^1_{\frak{X}/\cur{V}}$ for $\frak{X}$ a formal scheme of finite type over $\cur{V}\pow{t}$. From now on we will exclusively work with rigid spaces in the sense of Fujiwara/Kato, note that in particular this means that we can make sense of the fibred product $\sh{X}\times_K\sh{X}$ when $\sh{X}$ is locally of finite type over $\D^b_K$. 

\begin{lem} Let $\frak{X}$ be a $\cur{V}$-flat formal scheme of finite type over $\cur{V}\pow{t}$, so that we have $\Omega^1_{\frak{X}/\cur{V}}=\varprojlim_n \Omega^1_{X_n/\cur{V}_n}$ with $X_n$ the mod $\pi^{n+1}$-reduction of $\frak{X}$. Let $\frak{I}$ denote the ideal of the diagonal of $\frak{X}$ inside $\frak{X}\times_{\cur{V}}\frak{X}$. Then there is an isomorphism 
$$ \Omega^1_{\frak{X}/\cur{V}} \cong \frak{I}/\frak{I}^2
$$
compatible with the natural derivations of $\cur{O}_{\frak{X}}$ into both. 
\end{lem}

\begin{proof} The question is local on $\frak{X}$, which we may therefore assume to be affine, $\frak{X}\cong \spf{A}$. Let $\hat{I}=\ker (A\widehat{\otimes}_\cur{V} A\rightarrow A)$, note that there is a natural continuous derivation $d:A\rightarrow \hat{I}/\hat{I}^2$, and hence we get a linear map
$$ \widehat{\Omega}^1_{A/\cur{V}}\rightarrow \hat{I}/\hat{I}^2.
$$
Since $\Omega^1_{A/\cur{V}}=\varprojlim_n \Omega^1_{A_n/\cur{V}_n}$, it suffices to prove that $I/I^2$ is $p$-adically separated and complete, and that we can identify $I/I^2 \otimes_\cur{V} \cur{V}_n$ with $I_n/I_n^2$, where $I_n$ is the kernel of $A_n\otimes_{\cur{V}_n} A_n \rightarrow A_n$. To see this, first note that the exact sequence
$$ 0 \rightarrow I \rightarrow A\widehat{\otimes}_\cur{V} A \rightarrow A \rightarrow 0
$$
together with $\cur{V}$-flatness of $A$ implies that $I\otimes_{\cur{V}} \cur{V}_n \cong I_n$, and hence that $I\cong \varprojlim_n I_n$. It then follows easily that $\hat{I}^2\otimes_\cur{V} \cur{V}_n\twoheadrightarrow I_n^2$. Let $A^{(2)}$ (resp. $A_n^{(2)}$) denote $(A\widehat{\otimes}_\cur{V} A)/\hat{I}^2$ (resp. $(A_n\otimes_{\cur{V}_n} A_n)/I_n^2$), then we have a commutative diagram
$$ \xymatrix{ \hat{I}^2 \otimes_\cur{V}\cur{V}_n \ar[r]\ar@{->>}[d] & A_n\otimes_{\cur{V}_n} A_n \ar[r]\ar@{=}[d] & A^{(2)} \otimes_\cur{V} \cur{V}_n \ar[d] \ar[r] & 0 \\
I_n^2 \ar[r] & A_n\otimes_{\cur{V}_n}A_n \ar[r] & A_n^{(2)} \ar[r] & 0 }
$$
with exact rows, this implies that $A^{2}\otimes_\cur{V}\cur{V}_n\cong A_n^{(2)}$. Hence by reducing the exact sequence
$$ 0 \rightarrow \hat{I}/\hat{I}^2 \rightarrow A^{(2)}\rightarrow A\rightarrow 0
$$
moduli $\pi^{n+1}$, using $\cur{V}$-flatness of $A$, and taking the limit as $n\rightarrow \infty$ gives the required result. 
\end{proof}

Hence we define, for any rigid variety $\sh{X}$ over $S_K$, the module of differentials
$$ \Omega^1_{\sh{X}/K}:=\cur{I}/\cur{I}^2
$$
where $\cur{I}\subset \cur{O}_{\sh{X}\times_K\sh{X}}$ is the ideal of the diagonal. By the previous lemma, whenever $\sh{X}$ is the generic fibre of some $\cur{V}$-flat formal scheme $\frak{X}$ of finite type over $\cur{V}\pow{t}$ we have
$$ \Omega^1_{\sh{X}/K} = \Omega^1_{\frak{X}/\cur{V}}  \otimes_{\cur{O}_\frak{X}} \cur{O}_\sh{X}
$$
compatibly with the natural derivations
\begin{align*} d:\cur{O}_\frak{X} &\rightarrow \Omega^1_{\frak{X}/\cur{V}} \\
d:\cur{O}_\sh{X}&\rightarrow \Omega^1_{\sh{X}/K}
\end{align*}
and hence in particular we get the following.
\begin{cor} \label{diffexact} For any rigid variety $\sh{X}$ over $S_K$, $\Omega^1_{\sh{X}/K}$ is a coherent $\cur{O}_\sh{X}$-module, and there is a natural exact sequence
$$ \cur{O}_\sh{X}\cdot dt \rightarrow \Omega^1_{\sh{X}/K}\rightarrow \Omega^1_{\sh{X}/S_K}\rightarrow 0
$$
which is moreover exact on the left whenever $\sh{X}$ is smooth over $\D^b_K$.
\end{cor}

\begin{proof} The only thing not clear is the last claim about left exactness. To see this, we note that the question is local, and hence using Corollary 1.6.10 and Proposition 1.7.1 of \cite{huber2} we may assume that $\sh{X}\cong \spa{A}$ is affinoid with $A$ of the form 
$$ A \cong \frac{S_K\tate{x_1,\ldots,x_n,y_1,\ldots,y_m}}{(f_1,\ldots,f_m)}
$$
and $\det\left( \frac{\partial f_i}{\partial y_j} \right)$ a unit in $A$. From the corresponding statement on formal schemes over $\cur{V}\pow{t}$, we can see that $\Omega^1_{\sh{X}/K}$ can be described explicitly as the $\cur{O}_\sh{X}$-module associated to the $A$-module
$$ \frac{A\cdot dt \oplus \bigoplus_{i=1}^n A \cdot dx_i \oplus \bigoplus_{j=1}^m A\cdot dy_j}{( df_1,\ldots,df_m)}.
$$  
Since $\det\left( \frac{\partial f_i}{\partial y_j} \right)\in A^\times$, it follows that this is equal to 
$$ A\cdot dt \oplus \bigoplus_{i=1}^n A \cdot dx_i
$$
and hence the claim is clear. \end{proof}

We now fix a rigid variety $\sh{X}$, smooth over $\D^b_K$. As usual, we set $\sh{X}^{(n)}$ to be the $n$th-order infinitesimal neighbourhood of $\sh{X}$ inside $\sh{X}\times_K\sh{X}$, that is the closed subscheme of $\sh{X}\times_K \sh{X}$ defined by the ideal $\cur{I}^{n+1}$ according to II.7.3.5 of \cite{rigidspaces}, and $p_i^{(n)}:\sh{X}^{(n)}\rightarrow \sh{X}$ the two projections for $i=1,2$. 

\begin{defn} Let $\sh{X}$ be as above.
\begin{enumerate} \item An integrable connection on an $\cur{O}_\sh{X}$-module $\sh{E}$, relative to $K$, is a homomorphism of sheaves
$$ \nabla:\sh{E}\rightarrow \sh{E}\otimes \Omega^1_{\sh{X}/K}
$$
such that $\nabla(fe)=e\otimes df + f\nabla(e)$ for all $f\in \cur{O}_\sh{X}$ and $e\in \sh{E}$, and such that the induced map
$$ \nabla^2:\sh{E}\rightarrow \Omega^2_{\sh{X}/K}:=\Lambda^2\Omega^1_{\sh{X}/K}
$$
is zero.
\item A stratification on an $\cur{O}_\sh{X}$-module $\sh{E}$, relative to $K$, is a collection of compatible isomorphisms
$$ \epsilon_n:p_2^{(n)*}\sh{E}\isomto p_1^{(n)*}\sh{E}
$$
such that $\epsilon_0=\mathrm{id}$ and $p_{12}^*(\epsilon_n)p_{23}^*(\epsilon_n) = p_{13}^*(\epsilon)$ for all $n$.
\end{enumerate}
\end{defn}

The proof of the following result is then exactly the same as in the classical case.

\begin{prop} There is a natural equivalence of categories between $\cur{O}_\sh{X}$-modules with stratification relative to $K$ and $\cur{O}_\sh{X}$-modules with integrable connection relative to $K$.
\end{prop}

\subsection{Overconvergent isocrystals and connections}

Now that the theory of connections and stratifications relative to $K$ is in place, we can proceed to construct the category $\mathrm{Isoc}^\dagger(X/K)$ associated to a $k\lser{t}$-variety $X$ in entirely the same manner as in \S5 of \cite{rclsf1}. We begin with the following, somewhat ad-hoc definition.

\begin{defn} We say that a $p$-adic formal scheme $\frak{X}$ over $\cur{V}$ is of pseudo finite type if there exists a $\cur{V}$-morphism $$\frak{X} \rightarrow \spf{V\pow{t}}\times_\cur{V} \ldots \times_\cur{V}\spf{\cur{V}\pow{t}}$$ of $p$-adic formal schemes which is of finite type. We say that
rigid space $\sh{X}$ over $K$ is locally of pseudo finite type if there exists a $K$-morphism 
$$ \sh{X} \rightarrow \mathbb{D}^b_K \times_K \ldots \times_K \mathbb{D}^b_K
$$
of rigid spaces over $K$ which is locally of finite type. Morphisms are morphisms over $\cur{V}$ and $K$ respectively, note that the generic fibre of a formal scheme of pseudo finite type over $\cur{V}$ is locally of pseudo finite type over $K$, and that both categories are closed under taking fibre products in the larger categories of $p$-adic formal schemes over $\cur{V}$ and rigid spaces over $K$ respectively.
\end{defn}

\begin{defn} The category of $k\lser{t}$-frames over $\cur{V}$ consists of triples $ (X,Y,\frak{P})
$ consisting of an open immersion $X\rightarrow Y$ of a $k\lser{t}$-variety into a $k\pow{t}$-scheme, and a closed immersion $Y\rightarrow \frak{P}$ of $Y$ into a formal $\cur{V}$-scheme which is of pseudo finite type. We say that $(X,Y,\frak{P})$ is proper if $Y$ is proper over $k\pow{t}$ and smooth if $\frak{P}$ is formally smooth over $\cur{V}$ in a  neighbourhood of $X$. Morphisms of frames are morphisms over $(k\lser{t},k\pow{t},\cur{V})$, therefore the category of $k\lser{t}$-frames over $\cur{V}$ contains the category of $\cur{V}\pow{t}$-frames (\cite{rclsf1}, Definition 2.1) as a non-full subcategory.
\end{defn}

If $(X,Y,\frak{P})$ is a $k\lser{t}$-frame over $\cur{V}$, then we have the specialisation map
$$ \frak{P}_K \rightarrow P = \frak{P}\times_\cur{V} k
$$
and can define the tubes $]X[_\frak{P}$ and $]Y[_\frak{P}$ exactly as in \S2 of \cite{rclsf1}, together with the map $j:]X[_\frak{P}\rightarrow ]Y[_\frak{P}$ and the functor $j_X^\dagger=j_*j^{-1}$. We can therefore make the following definitions.

\begin{defn} \begin{enumerate} \item An overconvergent isocrsytal on a $k\lser{t}$-variety $X$ is a collection $\{\sh{E}_\frak{Q}$\} of $j_U^\dagger\cur{O}_{]W[_\frak{Q}}$-modules, one for each pair consisting of a $k\lser{t}$-frame $(U,W,\frak{Q})$ over $\cur{V}$ and a map $U\rightarrow X$, together with isomorphisms $u^*\sh{E}_\frak{Q}\rightarrow \sh{E}_{\frak{Q}'}$ associated to any morphism of frames $(U',W',\frak{Q}')\rightarrow (U,W,\frak{Q})$ commuting with the given morphisms to $X$. The category of such objects is denoted $\mathrm{Isoc}^\dagger(X/K)$.
\item If $(X,Y,\frak{P})$ is a $k\lser{t}$-frame over $\cur{V}$, then an overconvergent isocrsytal on $(X,Y,\frak{P})$ is a collection $\{\sh{E}_\frak{Q}$\} of $j_U^\dagger\cur{O}_{]W[_\frak{Q}}$-modules, one for each $k\lser{t}$-frame $(U,W,\frak{Q})$ over $\cur{V}$ mapping to $(X,Y,\frak{P})$, together with isomorphisms $u^*\sh{E}_\frak{Q}\rightarrow \sh{E}_{\frak{Q}'}$ associated to any morphism of frames $(U',W',\frak{Q}')\rightarrow (U,W,\frak{Q})$ such that the diagram
$$\xymatrix{  (U',W') \ar[r]\ar[dr] & (U,W) \ar[d] \\ & (X,Y)
}
$$
commutes. The category of such objects is denoted $\mathrm{Isoc}^\dagger((X,Y,\frak{P})/K)$.
\end{enumerate} \end{defn}

Then exactly the same argument as in \S5 of \cite{rclsf1} shows that if $(X,Y,\frak{P})$ is a smooth and proper $\cur{V}\pow{t}$-frame, i.e. $\frak{P}$ is smooth over $\cur{V}\pow{t}$ around $X$, then the forgetful functor
$$ \mathrm{Isoc}^\dagger(X/K)\rightarrow \mathrm{Isoc}^\dagger((X,Y,\frak{P})/K)
$$
is an equivalence of categories. Again, if $(X,Y,\frak{P})$ is a smooth frame over $\cur{V}\pow{t}$, we will let
$$ p_i: (X,Y,\frak{P}\times_\cur{V} \frak{P}) \rightarrow (X,Y,\frak{P})
$$
denote the natural projections. This clashes with the notation used in \emph{loc. cit.}, this should hopefully not cause too much confusion.

\begin{defn} \label{ocrelstrat}Let $(X,Y,\mathfrak{P})$ be a smooth frame over $\cur{V}\pow{t}$.  \begin{enumerate} 
\item An overconvergent stratification (relative to $K$) on a $j_X^\dagger\cur{O}_{]Y[_\mathfrak{P}}$-module $\sh{E}$ is an isomorphism
$$
\epsilon:p_2^*\sh{E}\rightarrow p_1^*\sh{E}
$$
of $j^\dagger_X\cur{O}_{]Y[_{\frak{P}\times_\cur{V} \frak{P}}}$-modules, called the Taylor isomorphism, such that $\Delta^*(\epsilon)=\mathrm{id}$ and $p_{13}^*(\epsilon)=p_{12}^*(\epsilon)\circ p_{23}^*(\epsilon)$ on $]Y[_{\frak{P}\times_\cur{V} \frak{P}\times_\cur{V}\frak{P}}$. The category of coherent $j_X^\dagger\cur{O}_{]Y[_\mathfrak{P}}$-modules with overconvergent stratification is denoted $\mathrm{Strat}^\dagger((X,Y,\mathfrak{P})/K)$.
\item A stratification on a $j_X^\dagger\cur{O}_{]Y[_\frak{P}}$-module $\sh{E}$ is a stratification as an $\cur{O}_{]Y[_\frak{P}}$-module. The category of coherent $j_X^\dagger\cur{O}_{]Y[_\mathfrak{P}}$-modules with a stratifications as $\cur{O}_{]Y[_\frak{P}}$-modules is denoted $\mathrm{Strat}((X,Y,\mathfrak{P})/K)$.
\item An integrable connection on a $j_X^\dagger\cur{O}_{]Y[_\mathfrak{P}}$-module $\sh{E}$ is an integrable connection as an $\cur{O}_{]Y[_\frak{P}}$-module. The category of coherent $j_X^\dagger\cur{O}_{]Y[_\mathfrak{P}}$-modules with an integrable connection is denoted $\mathrm{MIC}((X,Y,\mathfrak{P})/K)$
\end{enumerate}
\end{defn}

Again, exactly as in \cite{rclsf1}, we have a series of functors
\begin{align*} \mathrm{Isoc}^\dagger(X/K)\rightarrow \mathrm{Isoc}^\dagger((X,Y,\mathfrak{P})/K)&\rightarrow \mathrm{Strat}^\dagger((X,Y,\mathfrak{P})/K) \\ &\rightarrow \mathrm{Strat}((X,Y,\mathfrak{P})/K) \rightarrow \mathrm{MIC}((X,Y,\mathfrak{P})/K) 
\end{align*}
with everything except $\mathrm{Strat}^\dagger((X,Y,\frak{P})/K) \rightarrow  \mathrm{Strat}((X,Y,\frak{P})/K)$ being an equivalence of categories. 

\begin{prop} The functor $\mathrm{Strat}^\dagger((X,Y,\frak{P})/K) \rightarrow  \mathrm{Strat}((X,Y,\frak{P})/K)$ is fully faithful.
\end{prop}

\begin{proof} Faithfulness is clear, since both categories embed faithfully in the category of coherent $j_X^\dagger\cur{O}_{]Y[_\frak{P}}$-modules. We are therefore given a morphism
$$ \cur{E}\rightarrow \cur{F}
$$
between coherent $j^\dagger_X\cur{O}_{\frak{P}_K}$-modules with overconvergent stratification, which commutes with the level $n$ Taylor isomorphisms
\begin{align*} p_2^{(n)*}\cur{E} &\rightarrow p_1^{(n)*}\cur{E} \\ 
p_2^{(n)*}\cur{F} &\rightarrow p_1^{(n)*}\cur{F}
\end{align*}
for all $n$, and we must show that it commutes with the overconvergent Taylor isomorphisms
\begin{align*} p_2^{*}\cur{E} &\rightarrow p_1^{*}\cur{E}
\\ p_2^{*}\cur{F} &\rightarrow p_1^{*}\cur{F}.
\end{align*}
As in the proof of Proposition 7.2.8 of \cite{rigcoh}, it therefore suffices to show that the map
$$ \mathrm{Hom}_{j_X^\dagger\cur{O}_{\frak{P}^2}}(p_2^*\cur{E},p_1^*\cur{F}) \rightarrow \varprojlim_n \mathrm{Hom}_{\cur{O}_{\frak{P}^{(n)}}}(p_2^{(n)*}\cur{E},p_1^{(n)*}\cur{F})
$$
is injective. By choosing a presentation of $\cur{E}_1$, it thus suffices to show that the natural map
$$ p_{1*}p_1^*\cur{F} \rightarrow \varprojlim p_{1}^{(n)*}\cur{F}
$$
of sheaves is injective. This question is local on $\frak{P}$, which we may therefore assume to be affine, and by restricting to the interior tube $]X[^\circ_\frak{P}$ as in the proof of Proposition 5.7 of \cite{rclsf1} we may in fact assume that $\frak{P}$ is a formal scheme over $\cur{O}_{\ek}$ and that $Y=X$. Again, using the fact that the question is local on $\frak{P}_K$ it suffices to prove the injection on global sections. We may also choose co-ordinates 
$$ (x_1,\ldots,x_d):\frak{P} \rightarrow \widehat{\A}^d_{\cur{O}_{\ek}}
$$
which are \'{e}tale in a neighbourhood of $X$, so that $dt,dx_1,\ldots,dx_d$ is a basis for $\Omega^1_{]X[_\frak{P}/K}$. Hence, if we set $A=\Gamma(]X[_\frak{P},\cur{O}_{\frak{P}_K})$ and $M=\Gamma(]X[_\frak{P},\cur{F})$, we have isomorphisms
\begin{align*}  \Gamma(]X[_\frak{P},p_{1}^{(n)*}\cur{F}) &\cong M\otimes_A \frac{A[\tau,\chi_1,\ldots,\chi_d]}{(\tau,\chi_1,\ldots,\chi_d)^n} \\
\Gamma(]X[_\frak{P},\varprojlim_n p_{1}^{(n)*}\cur{F}) &\cong M\otimes_A A\pow{\tau,\chi_1,\ldots,\chi_d}
\end{align*}
where $\tau=t\otimes1-1\otimes t$ and $\chi_i=x_i\otimes1-1\otimes x_i$. Of course, the same calculation holds on any open subspace of $]X[_\frak{P}$, for example the closed tubes $[X]_m$ of radius $r^{-1/m}$ inside $\frak{P}_K$. If we let $A_m=\Gamma([X]_m,\cur{O}_{\frak{P}_K})$ then we therefore get 
$$ \Gamma([X]_m,\varprojlim_n p_{1}^{(n)*}\cur{F}) \cong M\otimes_A A_n\pow{\tau,\chi_1,\ldots,\chi_d}.
$$
The functions $\chi_1,\ldots,\chi_d$ induce a morphism
$$ \frak{P}\times_{\cur{V}}\frak{P} \rightarrow \widehat{\A}^d_{\frak{P}\times_\cur{V} \spf{\cur{V}\pow{t}}}= \widehat{\A}^d_\frak{P} \times_\cur{V} \spf{\cur{V}\pow{t}}
$$
which is \'{e}tale in a  neighbourhood of $X$, and hence by the analogue of Proposition 4.1 of \cite{rclsf1} over $\cur{V}\pow{t}\widehat{\otimes}_\cur{V}\cur{V}\pow{t}$ (whose proof is identical), we may replace $\frak{P}\times_\cur{V}\frak{P}$ by 
$$\widehat{\A}^d_\frak{P} \times_\cur{V} \spf{\cur{V}\pow{t}}= \widehat{\A}^d_\frak{P} \times_{\cur{V}\pow{t}} \spf{\cur{V}\pow{t}}\times_\cur{V} \spf{\cur{V}\pow{t}}
$$
Now, in this case the tubes $]X[$ are actually rigid varieties, and it is not hard to see that their formation commutes with fibre products, and hence we may make the identification
$$ ]X[_{\frak{P}\times_\cur{V}\frak{P}} \cong ]X[_{\widehat{\A}^d_\frak{P}} \times_{\D^b_K} ]\spec{k\pow{t}}[_{\spf{\cur{V}\pow{t}}\times_\cur{V} \spf{\cur{V}\pow{t}}}.
$$
Since we are only interested in global section, we may replace the open tubes $]X[$ and $]\spec{k\pow{t}}[$ by the corresponding closed tubes of radius $r^{-1/m}$ (where $r=\norm{1/\pi}$ for any uniformiser $\pi$). If we let $\cur{O}_m$ be the ring of functions on the closed tube $[\spec{k\pow{t}}]_m$ inside $\D^b_K\times_K \D^b_K$ then using Proposition 4.3.6 of \cite{rigcoh} we may identify the group of sections of $p_1^*\sh{F}$ on the closed tube $[X]_m$ inside $\frak{P}\times_\cur{V}\frak{P}$ with
$$  \cur{O}_m \widehat{\otimes}_{S_K} \left( M\otimes_A A_m \tate{r^{1/m}\chi_1,\ldots,r^{1/m}\chi_d} \right).
$$
By inducting on the number of generators of $M$, we may reduce to the case when $M$ is a quotient of $A$, and by replacing $A_m$ by the affinoid algebra $M\otimes_A A_m$ it suffices to prove that the map
$$  \cur{O}_m \widehat{\otimes}_{S_K} B \tate{r^{1/m}\chi_1,\ldots,r^{1/m}\chi_d} \rightarrow B\pow{\tau,\chi_1,\ldots,\chi_d}
$$
is injective for any topologically finite type $\ek$-algebra $B$. Here the map is the obvious one on $B \tate{r^{1/m}\chi_1,\ldots,r^{1/m}\chi_d}$ and on $\cur{O}_m$ is induced by the map
\begin{align*} S_K \widehat{\otimes}_K S_K &\rightarrow S_K \pow{\tau} \\
f \otimes g &\mapsto f(t)g(\tau+t) 
\end{align*}
(in other words $\tau=1\otimes t-t\otimes 1$). Hence it suffices to prove that the map
$$ \cur{O}_m \widehat{\otimes}_{S_K} B \rightarrow B\pow{\tau}
$$
is injective. But now by using adic flatness of $B$ over $\ek$, and writing everything out explicitly, it suffices to show that
$$ \frac{(S_K \widehat{\otimes}_K S_K) \tate{T}}{pT-(1\otimes t-t\otimes 1)^m } \rightarrow S_K\pow{\tau}
$$
is injective, which follows from a direct calculation.
\end{proof}

Hence we may define the full subcategory 
$$ \mathrm{MIC}^\dagger((X,Y,\frak{P})/K)\subset \mathrm{MIC}((X,Y,\frak{P})/K)
$$
as the essential image of $\mathrm{Strat}^\dagger((X,Y,\frak{P})/K)$, and we get an equivalence
$$ \mathrm{Isoc}^\dagger(X/\ekd)\cong  \mathrm{MIC}^\dagger((X,Y,\frak{P})/K)
$$
whenever $(X,Y,\frak{P})$ is a smooth $\cur{V}\pow{t}$-frame. Functoriality and Frobenius structures can be dealt with \emph{exactly} as in \emph{loc. cit.}, we therefore get categories $F\text{-}\mathrm{Isoc}^\dagger(X/K)$ which are Zariski-local in $X$, and functorial in $X$, $k\lser{t}$ and $K$. Since the category of $\cur{V}\pow{t}$-frames embeds faithfully in the category of $k\lser{t}$-frames over $\cur{V}$, we get canonical `forgetful' functors
$$ (F\text{-})\mathrm{Isoc}^\dagger(X/K) \rightarrow (F\text{-})\mathrm{Isoc}^\dagger(X/\ekd)
$$
which is seen to be faithful by comparison with the category of modules with integrable connection.

\begin{defn} Let $\cur{E}\in(F\text{-})\mathrm{Isoc}^\dagger(X/K)$. Then we define the rigid cohomology
$$ H^i_\rig(X/\ekd,\cur{E})
$$
to be the cohomology of the associated object in $(F\text{-})\mathrm{Isoc}^\dagger(X/\ekd)$.
\end{defn}

\subsection{The Gauss--Manin connection and $(\varphi,\nabla)$-modules}

We can now explain how for every variety $X/k\lser{t}$, and every object $\sh{E}\in \mathrm{Isoc}^\dagger(X/K)$, there is a natural connection on $H^i_\rig(X/\ekd,\sh{E})$.

\begin{prop} Let $(X,Y,\mathfrak{P})$ be a smooth $\cur{V}\pow{t}$-frame. Then $j^\dagger_X \Omega^1_{]Y[_\mathfrak{P}/K}$ is a locally free $j_X^\dagger\cur{O}_{]Y[_\mathfrak{P}}$-module, and there is an exact sequence
$$
0\rightarrow j_X^\dagger\cur{O}_{]Y[_\mathfrak{P}}\cdot dt \rightarrow j^\dagger_X \Omega^1_{]Y[_\mathfrak{P}/K} \rightarrow j^\dagger_X \Omega^1_{]Y[_\mathfrak{P}/S_K} \rightarrow 0
$$
of locally free $j_X^\dagger\cur{O}_{]Y[_\mathfrak{P}}$-modules. 
\end{prop}

\begin{proof}Follows from Corollary \ref{diffexact}. 
\end{proof}

The point of this is that if $(X,Y,\mathfrak{P})$ is a smooth $\cur{V}\pow{t}$-frame, and $\sh{E}$ is a $j_X^\dagger\cur{O}_{]Y[}$-module $\sh{E}$ with integrable connection relative to $K$, the we obtain a Gauss--Manin connection
$$
\nabla: H^i(]Y[_\mathfrak{P},\sh{E}\otimes \Omega^*_{]Y[_\mathfrak{P}/S_K}) \rightarrow H^i(]Y[_\mathfrak{P},\sh{E}\otimes \Omega^*_{]Y[_\mathfrak{P}/S_K})
$$
on its de Rham cohomology (over $S_K$) in the usual way as follows. Define a filtration $F^\bu$ on the complex $\sh{E}\otimes \Omega^*_{]Y[_\mathfrak{P}/K}$ by
$$
 F^0=\sh{E}\otimes  \Omega^*_{]Y[_\mathfrak{P}/K},\; 
 F^1= \mathrm{im} (\sh{E} \otimes \Omega^{*-1}_{]Y[_\mathfrak{P}/K} \otimes \cur{O}_{]Y[_\mathfrak{P}}\cdot dt\rightarrow \sh{E}\otimes \Omega^*_{]Y[_\mathfrak{P}/K}) ,\;
 F^2= 0 
$$
so that we have
$$
\mathrm{Gr}_F^0 = \sh{E}\otimes \Omega^*_{]Y[_\mathfrak{P}/S_K},\;
\mathrm{Gr}_F^1 \cong \sh{E} \otimes \Omega^{*-1}_{]Y[_\mathfrak{P}/S_K} \cdot dt.
$$
Thus the spectral sequence associated to this filtration has $E_1$-page with differentials
$$ H^i(]Y[_\mathfrak{P}, \sh{E}\otimes\Omega^*_{]Y[_\mathfrak{P}/S_K}) \rightarrow H^i(]Y[_\mathfrak{P}, \sh{E}\otimes\Omega^*_{]Y[_\mathfrak{P}/S_K})\cdot dt
$$
which defines the Gauss--Manin connection on each $H^i(]Y[_\mathfrak{P}, \sh{E}\otimes\Omega^*_{]Y[_\mathfrak{P}/S_K})$.

It is easy to see that if $u:(X',Y',\frak{P}')\rightarrow (X,Y,\frak{P})$ is a morphism of smooth and proper $\cur{V}\pow{t}$-frames, and $\sh{E}\in\mathrm{MIC}((X,Y,\frak{P})/K)$ then the natural morphism
$$ H^i(]Y[_\frak{P},\sh{E}\otimes\Omega^*_{]Y[_\frak{P}}) \rightarrow   H^i(]Y'[_{\frak{P}'},u^*\sh{E}\otimes\Omega^*_{]Y'[_{\frak{P}'}})
$$
is horizontal, which implies firstly that the connection does not depend on the choice of frame $(X,Y,\frak{P})$ but only on $X$, and secondly that the functoriality morphisms in rigid cohomology
$$H^i_\rig(X/\ekd,\cur{E})\rightarrow  H^i_\rig(X'/\ekd,f^*\sh{E})
$$
associated to a pair of morphisms $f:X'\rightarrow X$ and $f^*\cur{E}\rightarrow \cur{E}'$ are in fact horizontal.

To extend this construction to compactly supported cohomology is entirely straightforward, to explain how this is done we fix a smooth and proper $\cur{V}\pow{t}$-frame $(X,Y,\frak{P})$ and $\cur{E}\in\mathrm{MIC}((X,Y,\frak{P})/K)$. Then exactly as in 5.13 of \cite{rclsf1}, we may choose some neighbourhood $V$ of $]X[_\frak{P}$ inside $]Y[_\frak{P}$ to which $\cur{E}$ extends. Then there is a natural Gauss--Manin filtration on $\cur{E}\otimes\Omega^*_{V/K}$, and hence we can consider $\mathbf{R}\underline{\Gamma}_{]X[_\frak{P}}(\cur{E}\otimes\Omega^*_{V/K})$ as an object in the filtered derived category of sheaves on $V$. Exactly as in the case for cohomology without supports, looking at the differentials on the $E_1$-page of the associated spectral sequence calculating
$$ H^i(V,\mathbf{R}\underline{\Gamma}_{]X[_\frak{P}}(\cur{E}\otimes\Omega^*_{V/K}))
$$
then gives morphisms
$$  H^i_{c,\rig}((X,Y,\frak{P})/\ekd,\cur{E}) \rightarrow  H^i_{c,\rig}((X,Y,\frak{P})/\ekd,\cur{E})\cdot dt
$$
which define the connection on cohomology groups with compact support. Again, one easily verifies that these only depend on $X$ and not on the choice of frame, and that the functoriality morphisms associated to either proper maps or open immersions are horizontal. It is also straightforward to check that the Poincar\'{e} pairing
$$ H^i_\rig (X/\ekd,\cur{E})\times H^j_{c,\rig}(X/\ekd,\cur{F})\rightarrow H^{i+j}_{c,\rig}(X/\ekd,\cur{E}\otimes \cur{F})
$$
constructed in \S3 of \cite{rclsf2} is compatible with the connections.

Since the Frobenius action on both $H^i_\rig(X/\ekd,\cur{E})$ and $H^i_{c,\rig}(X/\ekd,\cur{E})$ (when $\cur{E}\in F\text{-}\mathrm{Isoc}^\dagger(X/\ekd)$) is simply that induced by functoriality, it follows from functoriality of the Gauss--Manin connection that the connection and Frobenius structures are compatible, in the sense that the diagram
$$ \xymatrix{ H^i_\rig(X/\ekd,\cur{E})\ar[r]^\nabla \ar[d]^\varphi & H^i_\rig(X/\ekd,\cur{E})\ar[d]^{\partial_t(\sigma(t))\varphi} \\
H^i_\rig(X/\ekd,\cur{E}) \ar[r]^\nabla & H^i_\rig(X/\ekd,\cur{E})
} 
$$
commutes.

\begin{cor}\label{pnmodst} Let $X/k\lser{t}$ and $\sh{E}\in F\text{-}\mathrm{Isoc}^\dagger(X/K)$. Assume that $H^i_\rig(X/\ekd,\sh{E})$ is finite dimensional, and that the the induced linear Frobenius
$$  \varphi^\sigma: H^i_\rig(X/\ekd,\sh{E}) \otimes_{\ekd,\sigma} \ekd\rightarrow H^i_\rig(X/\ekd,\sh{E})
$$
is bijective. Then $H^i_\rig(X/\ekd,\sh{E})$ is a $(\varphi,\nabla)$-module over $\ekd$. In particular this is the case if the base change morphism
$$H^i_\rig(X/\ekd,\sh{E}) \otimes_{\ekd} \ek \rightarrow H^i_\rig(X/\ek,\hat{\sh{E}})
$$ 
is an isomorphism. The same is true for compactly supported cohomology.
\end{cor}

\begin{cor} Let $X/k\lser{t}$ be a smooth curve, $\sh{E}\in F\text{-}\mathrm{Isoc}^\dagger(X/K)$ and $i\geq0$. Then $H^i_\rig(X/\ekd,\sh{E})$ is a $(\varphi,\nabla)$-module over $\ekd$, and if $\cur{E}$ extends to an overconvergent $F$-isocrystal (relative to $\ekd$) on a smooth compactification $\overline{X}$ of $X$ then $H^i_{c,\rig}(X/\ekd,\cur{E})$ is also a $\pn$-module. In the latter case the Poincar\'{e} pairing
$$ H^i_\rig(X/\ekd,\cur{E}) \times H^{2-i}_{c,\rig}(X/\ekd,\cur{E}^\vee)\rightarrow \ekd(-1)
$$
is a perfect pairing of $\pn$-modules. \end{cor}

\section{Weight monodromy and $\ell$-independence}\label{wmli}

As mentioned in the introduction, we expect that the cohomology theory $H^*_\rig(X/\ekd)$ should be more intimately connected with the arithmetic properties of $X$ than the usual rigid cohomology $H^*_\rig(X/\ek)$. In particular, it is using $H^*_\rig(X/\ekd)$ that we can formulate $p$-adic analogues of well-known conjectures and results concerning the $\ell$-adic cohomology of varieties over local fields. In this section we give some examples of the ways in which we expect to see the arithmetic properties of $X$ reflected in its cohomology $H^*_\rig(X/\ekd)$. In some sense these are really toy examples, and in fact don't use the full information contained in $H^*_\rig(X/\ekd)$, but only that which is preserved by base changing to $\rk$.

\begin{defn} For any smooth curve $X/k\lser{t}$ and any $\sh{E}\in F\text{-}\mathrm{Isoc}^\dagger(X/K)$ we define
$$ H^i_\rig(X/\cur{R}_K,\sh{E}):= H^i_\rig(X/\ekd,\sh{E})\otimes_{\ekd}\cur{R}_K
$$
by the results of previous sections, this is a free $(\varphi,\nabla)$-module over $\cur{R}_K$. We also define $H^i_\rig(X/\cur{R}_K)$ to be $H^i_\rig(X/\ekd)\otimes_{\ekd}\cur{R}_K$.
\end{defn}

In fact, a lot of our results do not depend on our construction of $H^i_\rig(X/\ekd)$, since as mentioned in the introduction to, Kedlaya in his thesis proved that when $X$ is smooth and proper over $k\lser{t}$, and $K=W(k)[1/p]$, then $H^i_\rig(X/\ek)$ is overconvergent, that is descends to a $(\varphi,\nabla)$-module over $\ekd$. This actually suffices to prove both $\ell$-independence and the weight monodromy conjecture in cohomological dimension $\leq 1$, at least for smooth and proper varieties. It is the general formalism of $\ekd$-valued rigid cohomology, however, that allows us to then easily deduce results about open curves from the statements in the smooth and proper case. Also, again as mentioned in the introduction, we expect that once a suitably robust formalism of $\ekd$-valued cohomology has been developed, it will provide the natural setting for the construction of a $p$-adic weight spectral sequence `explaining' the weight filtration on the cohomology of smooth and proper varieties with semistable reduction.

For the rest of this section, we will assume that $k$ is a finite field, and contrary to the assumptions of the rest of the paper, Frobenius $\sigma$ will be assumed to be the $p$-power Frobenius. We will let $q$ denote the cardinality of $k$. We will assume that $\cur{V}=W=W(k)$ is the ring of Witt vectors of $k$, we therefore have $K=W[1/p]$. We will also let $\cur{R}_K^+$ denote the plus part of the Robba ring, that is the ring of convergent functions on the open unit disc over $K$.

\subsection{Weil--Deligne representations and Marmora's construction}

One of the indications that $H^*_\rig(X/\ekd)$ is a good version of $p$-adic cohomology to look at over $k\lser{t}$ is the fact that by base changing to $\cur{R}_K$ as just described, we can attach Weil--Deligne representations to the cohomology of varieties $X/k\lser{t}$. This is done using Marmora's construction in \cite{marmora} of the Weil--Deligne representation attached to a $(\varphi,\nabla)$-module over $\cur{R}_K$, which we now recall.

Let us denote by $G_{k\lser{t}}$ and $G_k$ the absolute Galois groups of $k\lser{t}$ and $k$ respectively, we therefore have an exact sequence
$$ 0 \rightarrow I_{k\lser{t}} \rightarrow G_{k\lser{t}} \rightarrow G_k \rightarrow 0
$$
where $I_{k\lser{t}}$ is the inertia group. Define the Weil group $W_{k\lser{t}}$ to be the inverse image of $\langle\mathrm{Frob}_k \rangle \cong \Z \subset G_k \cong  \widehat\Z$, topologised so that $I_K$ is open. Here $\mathrm{Frob}_k$ refers to \emph{arithmetic} Frobenius, and we define the map $v: W_{k\lser{t}}\rightarrow \Z$ by $\mathrm{Frob}_k^{v(\sigma)}\equiv \sigma \text{ mod } I_k$. For any representation $V$ of $W_{k\lser{t}}$ we denote by $V(1)$ the representation on the same space but with the action of any $\sigma\in W_{k\lser{t}}$ multiplied by $q^{v(\sigma)}$. 

\begin{defn} Let $E$ be a field of characteristic zero. An $E$-valued Weil--Deligne representation over $k\lser{t}$ is a continuous representation of the Weil group $W_{k\lser{t}}$ in a finite dimensional, discrete $E$-vector space, together with a $G_{k\lser{t}}$-equivariant, nilpotent monodromy operator
$$ N: V(1)\rightarrow V.
$$
The category of $E$-valued Weil--Deligne representations is denoted $\mathrm{Rep}_{E}(\mathrm{WD}_{k\lser{t}})$. Note that since $I_{k\lser{t}}$ is a profinite group, it follows that it must act via a finite quotient on any Weil--Deligne representation.
\end{defn}

If $(V,N)$ is some Weil--Deligne representation over $k\lser{t}$, then there is a unique $W_{k\lser{t}}$-invariant filtration $M_\bu V$ on $V$, called the monodromy filtration, such that:
\begin{enumerate} 
\item $N(M_iV(1))\subset M_{i-2}V$;
\item  $N$ induces an isomorphism
$$  N^k: \mathrm{Gr}^M_kV(k) \isomto \mathrm{Gr}^M_{-k}V.
$$ 
\end{enumerate}

\begin{defn} Let $V$ be a Weil--Deligne representation over $k\lser{t}$. 
\begin{enumerate}
\item We say that $V$ is pure of weight $i$ if the eigenvalues of any lift to $W_{k\lser{t}}$ of $\mathrm{Frob}_k$ are Weil numbers of weight $-i$, i.e. are algebraic numbers all of whose complex embeddings have absolute value $q^{-i/2}$. 
\item We say that $V$ is quasi-pure of weight $i$ if for all $k$ the $k$th graded piece $\mathrm{Gr}_k^MV$ of the monodromy filtration are pure of weight $i+k$ in the above sense.
\end{enumerate}
\end{defn}

\begin{rem} \begin{enumerate} \item Since $I_{k\lser{t}}$ acts on $V$ by a finite quotient, the condition of being pure or quasi-pure can be checked on a single lift of arithmetic Frobenius.
\item What we have termed `quasi-pure of weight $i$' has been called by Scholl (unpublished) `pure of weight $i$', since it is conjecturally expected that the `geometric' weight filtration on the cohomology of varieties has graded pieces which are quasi-pure, rather than pure. We do not follow this terminology, since we want a way to distinguish pure and quasi-pure Weil--Deligne representations. 
\end{enumerate}
\end{rem}

Marmora's construction takes a $\pn$-module over $\rk$, and associates to it a Weil--Deligne representation with coefficients in $K^\mathrm{un}$, the maximal unramified extension of $K$, as follows.

Recall from \S2.2 of \cite{tsuzuki1} that for every finite separable extension $F/k\lser{t}$ there is a corresponding finite \'{e}tale map $\cur{R}_K \rightarrow \cur{R}_K^F$ where $\cur{R}_K^F$ is a copy of the Robba ring but with a different constant field and parameter. Define
$$\cur{B}_0 := \mathrm{colim}_{F} \cur{R}_K^F$$
to be the colimit of all $\cur{R}_K^F$ as $F$ runs over all finite extensions of $k\lser{t}$ within some fixed separable closure, and let $\cur{B}=\cur{B}_0[\log t]$, that is we adjoint a formal variable $\log t$, who's derivative is $\frac{dt}{t}$. Note that $\mathrm{G}_{k\lser{t}}$ acts naturally on $\cur{B}_0$, and hence on $\cur{B}$ by letting it act trivially on $\log t$, and there is a natural extension to $\cur{B}$ of the Frobenius $\sigma$ on $\rk$. $\cur{B}$ also admits a monodromy operator $N$, that is a $\cur{B}_0$-derivation such that $N(\log t)=1$.

By the $p$-adic local monodromy theorem, every $(\varphi,\nabla)$-module $M$ over $\cur{R}_K$ becomes trivial over $\cur{B}$, i.e. admits a basis of sections which are horizontal for the extended connection
$$ \nabla: M\otimes_{\cur{R}_K} \cur{B} \rightarrow M\otimes_{\cur{R}_K} \Omega^1_\cur{B}
$$
where $\Omega^1_\cur{B}=\cur{B}\cdot dt$. Note that we get a `diagonal' Frobenius structure on $M\otimes_{\cur{R}_K} \cur{B}$, as well as a monodromy operator and a $G_{k\lser{t}}$-action arising from the given monodromy and Galois actions on $\cur{B}$. Hence we get a functor
$$ M\mapsto \left(M\otimes_{\cur{R}_K} \cur{B}\right)^{\nabla=0}
$$
from the category of \emph{Deligne modules} over $K^\mathrm{un}$, that is vector spaces together with a monodromy operator, a Frobenius and a semilinear $G_{k\lser{t}}$-action, which satisfy certain natural compatibilities (see \S3.1 of \cite{marmora}). Composing with Fontaine's functor from Deligne modules to Weil--Deligne representations which `twists' the Galois action by the Frobenius to linearise it (see \S3.4 of \emph{loc. cit.}), we therefore get a functor
$$ \mathrm{WD}: \underline{\mathrm{M}\Phi}^\nabla_{\cur{R}_K} \rightarrow  \mathrm{Rep}_{K^\mathrm{un}}(\mathrm{WD}_{k\lser{t}}).
$$
from $\pn$-modules over $\cur{R}_K$ to $K^\mathrm{un}$-valued Weil--Deligne representations. In particular, we can talks about a $(\varphi,\nabla)$-module over $\cur{R}_K$ being pure or quasi-pure of some weight.

\subsection{The $p$-adic weight-monodromy conjecture in cohomological dimension 1}

We will now let $X/k\lser{t}$ be a smooth and proper variety, and denote by $H^i_\mathrm{cris}(X/\ek)$ its rational crystalline cohomology, as constructed by Ogus in \cite{ogus} (tensored from $\cur{O}_{\ek}$ to $\ek$) - this is a $(\varphi,\nabla)$ over $\ek$. Then Theorem 7.0.1 of \cite{kedthesis} says that this module is overconvergent, that is descends to a $(\varphi,\nabla)$-module $H^i_\mathrm{cris}(X/\ekd)$ over $\ekd$ (note that $H^i_\mathrm{cris}(X/\ekd)$ is well-defined by Kedlaya' full faithfulness result, i.e. Theorem 5.1 of \cite{kedfull}). Of course, conjecturally, we expect this to coincide with $H^i_\rig(X/\ekd)$, however, we currently only know that this is the case in dimension 1. 

\begin{defn} Define the $\cur{R}_K$-valued crystalline cohomology of $X$ by
$$ H^i_\mathrm{cris}(X/\cur{R}_K) = H^i_\mathrm{cris}(X/\ekd) \otimes_{\ekd} \cur{R}_K.
$$
\end{defn}

Then the main result of this section is then the following.

\begin{theo} \label{wmc} Let $X/k\lser{t}$ be a smooth and proper variety, of pure dimension $n$, and let $i\in \{0,1,2n-1,n\}$. Then $H^i_\mathrm{cris}(X/\cur{R}_K)$ is quasi-pure of weight $i$.
\end{theo}

\begin{cor} Let $X/k\lser{t}$ be a smooth and proper curve, and $i\geq0$. Then $H^i_\rig(X/\cur{R}_K)$ is quasi-pure of weight $i$.
\end{cor}

We start by recalling the notion of a Dieudonn\'{e} module over versions of the rings $S_K,\ekd,\ek$ in which $p$ is not invertible. 

\begin{defn} Let $R$ be one of the rings $W\pow{t},\cur{O}_{\ekd},\cur{O}_{\ek}$. Then a Dieudonn\'{e} module over $R$ is a finite free $R$ module $M$, together with a topologically quasi-nilpotent connection $\nabla:M\rightarrow M\otimes_R \Omega^1_R$ and horizontal morphisms
$ F: \sigma^*M\rightarrow M$, $V:M\rightarrow \sigma^*M$ such that $FV=p\mathrm{id}_M$ and $VF=p\mathrm{id}_{\sigma^*M}$. The category of such objects is denoted $\mathrm{\underline{MD}}_{R}$
\end{defn}

\begin{rem} Here $\Omega^1_R$ is simply defined to be the free $R$-module on $dt$, with canonical derivation $d:R\rightarrow \Omega^1_R$.
\end{rem}

Then thanks to the Main Theorem on page 6 of \cite{dej1}, we have equivalences of categories
\begin{align*} D:\mathrm{BT}_{k\pow{t}} &\isomto \mathrm{\underline{MD}}_{W\pow{t}} \\
D:\mathrm{BT}_{k\lser{t}} &\isomto \mathrm{\underline{MD}}_{\cur{O}_{\ek}} 
\end{align*}
between $p$-divisible groups over $k\pow{t}$ (resp. $k\lser{t}$) and Dieudonn\'{e} modules over $W\pow{t}$ (resp. $\cur{O}_{\ek}$), which commute with the natural base change functors.

To prove Theorem \ref{wmc}, note that by Poincar\'{e} duality in crystalline cohomology it suffices to treat the cases $i=0,1$, and the case $i=0$ is clear, hence the only real content is the case $i=1$. To deal with this case, we introduce the Albanese variety $A$ of $X$, so that there is an isomorphism
$$ H^1_\mathrm{cris}(X/\ek) \cong D(A)[1/p]^\vee
$$ 
between the first crystalline cohomology of $X$ and the dual of the rational Dieudonn\'{e} module of the $p$-divisible group of $A$. Theorem 7.01. of \cite{kedthesis} shows that $D(A)$ is oveconvergent, i.e. comes from a unique Dieudonn\'{e} module $D^\dagger(A)$ over $\cur{O}_{\ekd}=\cur{R}_K^\mathrm{int}$. Denote by $D_{\cur{R}}(A)$ its base change to $\cur{R}_K$. Then Theorem \ref{wmc} is a consequence of the following result. 

\begin{theo}\label{wmc2} The $(\varphi,\nabla)$-module $D_{\cur{R}}(A)$ over $\cur{R}_K$ is quasi-pure of weight $-1$.
\end{theo}

We will demonstrate this by explicitly describing the weight filtration on $D_\cur{R}(A)$ in the semistable case, using results from \cite{sga719}, and then show that this actually coincides with the monodromy filtration. Fist, however, we need a couple of technical results on $(\varphi,\nabla)$-modules and Dieudonn\'{e} modules.

\begin{lem} \label{deflem} \begin{enumerate}
\item Let $X$ be be an invertible $n\times n$ matrix over $\cur{R}_K$. Then there exist invertible $n\times n$ matrices $Y$ over $\ekd$ and $Z$ over $\cur{R}_K^+$ such that $X =YZ$.
\item Let $X$ be be an invertible $n\times n$ matrix over $\ek$. Then there exist invertible $n\times n$ matrices $Y$ over $\cur{O}_{\ek}$ and $Z$ over $K$ such that $X =YZ$.
\end{enumerate}
\end{lem}

\begin{proof} \begin{enumerate}
\item This is Proposition 6.5 of \cite{padicmono}. 
\item This is Proposition 2.18 of \cite{ambrusmono}\end{enumerate}

\end{proof}

\begin{prop} \label{defprop} \begin{enumerate}
\item Let $M$ be a $(\varphi,\nabla)$-module over $\ekd$, and suppose that $M\otimes_{\ekd} \cur{R}_K$ is defined over $\cur{R}_K^+$. Then $M$ is defined over $S_K$. 
\item Let $M$ be a Dieudonn\'{e} module over $\cur{O}_{\ek}$, and suppose that $M\otimes_{\cur{O}_{\ek}} \ek$ is defined over $S_K$. Then $M$ is defined over $W\pow{t}$. 
\end{enumerate}
\end{prop}

\begin{proof} This is Theorem 2.10 and Proposition 2.19 of \cite{ambrusmono}.
\end{proof}

Now, let $\cur{A}$ be the N\'{e}ron model of $A$ over $k\pow{t}$, with identity component $\cur{A}^\circ$ and special fibre $\cur{A}_0$. Let $\cur{A}_0^\circ$ denote the identity component of $ \cur{A}_0$. Since Theorem \ref{wmc} is insensitive to replacing $k\lser{t}$ by a finite extension, we may assume that $A$ has semistable reduction, in other words that $\cur{A}_0^\circ$ is an extension of an abelian variety $B_0$ over $k$ by a torus $T_0$. Let $\hat{\cur{A}}$ denote the formal completion of $\cur{A}$ along its special fibre, and $\hat{T}\subset \hat{\cur{A}}$ its maximal formal sub-torus, as in \S5.1 of \cite{sga719}. Write 
\begin{align*} G&=A[p^\infty] \\ G^f&= \hat{\cur{A}}[p^\infty]\otimes_{k\pow{t}}k\lser{t} \\ G^f&= \hat{T}[p^\infty]\otimes_{k\pow{t}}k\lser{t}
\end{align*} so that we have a filtration 
$$ 0\subset G^t\subset G^f\subset G
$$
of $p$-divisible groups over $k\lser{t}$, moreover, it follows from \S2.3.3 of \cite{sga719} that $G^f$ is the largest sub-$p$-divisible group of $G$ which extends to a $p$-divisible group over $k\pow{t}$. By de Jong's equivalence we therefore get a filtration
$$ 0 \subset D^t(A) \subset D^f(A)\subset D(A)
$$
of the associated Dieudonn\'{e} modules over $\cur{O}_{\ek}$, which descends to a filtration
$$ 0 \subset D^{\dagger,t}(A) \subset D^{\dagger,f}(A) \subset D^{\dagger}(A)
$$
of $(\varphi,\nabla)$-modules over $\cur{O}_{\ekd}$. Moreover, $D^f(A)$ is the largest submodule of $D(A)$ defined over $W\pow{t}$. We may now base change to $\cur{R}$ to get a filtration
$$ 0 \subset D_{\cur{R}}^t(A) \subset D_{\cur{R}}^f(A) \subset D_{\cur{R}}(A)
$$
of $(\varphi,\nabla)$-modules over $\cur{R}$.

\begin{prop} \label{key1} $D_{\cur{R}}^f(A)$ is the largest sub-$(\varphi,\nabla)$-module of $D_{\cur{R}}(A)$ defined over $\cur{R}_K^+$, i.e. by Dwork's trick it is the largest sub-$(\varphi,\nabla)$-module of $D_{\cur{R}}(A)$ admitting a basis of horizontal sections.
\end{prop}

\begin{proof} Suppose that $N\subset D_{\cur{R}}(A)$ is defined over $\cur{R}_K^+$. Then thanks to Corollaire 3.2.27 of \cite{marmora}, there exists a sub-$(\varphi,\nabla)$-module $N'\subset D^\dagger(A)[1/p]$ (over $\ekd$) such that $N'\otimes_{\ekd}\cur{R}_K \cong N$. By Proposition \ref{defprop}(i), $N'$ is actually defined as a $(\varphi,\nabla)$-module over $S_K$. The intersection $N'':=N'\cap D^\dagger(A)$ is then a sub-Dieudonn\'{e} module such that $N''[1/p]= N'$, by Proposition \ref{defprop}(ii), $N''$ is actually defined as a Dieudonn\'{e} module over $W\pow{t}$. Therefore we must have $N''\subset D^{\dagger,f}(A)$ and hence $N\subset D^f_{\cur{R}}(A)$ as claimed.
\end{proof}

\begin{prop} There are canonical isomorphisms
\begin{align*}  D_{\cur{R}}^t(A) &\cong H^1_\rig(T_0/K)^\vee \otimes_K \cur{R}_K \\
\frac{D_{\cur{R}}^f(A)}{D_{\cur{R}}^t(A)} &\cong H^1_\mathrm{rig}(B_0/K)^\vee \otimes_K \cur{R}_K
\end{align*}
of $(\varphi,\nabla)$-modules over $\cur{R}_K$. 
\end{prop}

\begin{proof} By the compatibility of the Dieudonn\'{e} module functor with base change, we know that $D^t(A) \otimes_{W\pow{t}} W\cong D(T_0)$, compatibly with Frobenius, where we have abused notation and written $D^t(A)$ to mean the canonical model of $D^t(A)$ as a Dieudonn\'{e} module over $W\pow{t}$. Hence we have, by Dwork's trick, that $D^t(A)\otimes_{W\pow{t}} \cur{R}_K^+ \cong D(T_0) \otimes_W \cur{R}_K^+$, and the first isomorphism follows from the fact that the rational Dieudonn\'{e} module of a torus is just the dual of its rigid cohomology. 

Similarly we have an isogeny 
$$ \frac{D^f(A)\otimes_{W\pow{t}}W }{D^t(A)\otimes_{W\pow{t}}W} \simeq D(B_0)
$$
and hence again using Dwork's trick we can see that 
$$ \frac{D^f(A)\otimes_{W\pow{t}} \cur{R}^+_K  }{D^t(A)\otimes_{W\pow{t}} \cur{R}_K^+ } \cong  D(B_0)\otimes_W\cur{R}^+_K
$$
and the second isomorphism follows.
\end{proof}

We denote this natural filtration on $D_{\cur{R}}(A)$ by $W_\bu$, so that $W_{-2}=D^t_{\cur{R}}(A)$, $W_{-1}=D^f_{\cur{R}}(A)$ and $W_0=D_{\cur{R}}(A)$. By the previous proposition, both $\mathrm{Gr}_{-2}^W$ and $\mathrm{Gr}_{-1}^W$ are constant $(\varphi,\nabla)$-modules over $\cur{R}_K$, and are pure of weights $-2$ and $-1$ respectively. Moreover, the orthogonality theorem, Th\'{e}or\`{e}me 5.2 of \cite{sga719} gives a canonical isomorphism
$$ \frac{D_{\cur{R}}(A)}{D^f_{\cur{R}}(A)} \cong \left( D^t_{\cur{R}}(A^\vee) \right)^\vee(1) 
$$
of $(\varphi,\nabla)$-modules over $\cur{R}_K$, where $A^\vee$ is the dual abelian variety of $A$. Thus by the previous proposition applied to $A^\vee$ we get that $\mathrm{Gr}_0^W$ is again constant, and pure of weight $0$. We therefore refer to $W_\bu$ as the weight filtration on $D_{\cur{R}}(A)$.

Hence to prove Theorem \ref{wmc2}, it suffices to prove that the weight filtration $W_\bu$ described in the pervious section coincides with the monodromy filtration $M_\bu$, up to a shift by $1$, and after applying the Marmora's functor $\mathrm{WD}$. We start by introducing the following notation.

\begin{defn} Let $M$ be a $(\varphi,\nabla)$-module over $\cur{R}$. Define the operator $D:M\rightarrow M$ by $D(m)=t\nabla(m)$. 
\end{defn}

The main ingredient is then the following explicit description of $D_{\cur{R}}(A)$. 

\begin{prop} \label{key2}There exists a basis $\{ e_1,\ldots,e_l,f_1\ldots,f_m,g_1,\ldots,g_n\}$ for $D_{\cur{R}}(A)$ such that:
\begin{enumerate} \item $\{e_1,\ldots,e_l\}$ is a basis for $W_{-2}$ and $\{e_1,\ldots,e_l,f_1,\ldots,f_m\}$ is a basis for $W_{-1}$;
\item $D(e_i)=D(f_j)=0$ for all $i,j$;
\item $D(g_k)$ is in the $K$-span of the $\{e_i\}$ for all $k$.
\end{enumerate}
\end{prop}

\begin{proof} For any $(\varphi,\nabla)$-module $M$ over $\cur{R}_K$, let $H^0(M)$ denote the subspace of horizontal sections, this is a finite dimensional vector space over $K$. Since $W_{-2}$ and $W_{-1}$ are both constant, i.e. admit bases of horizontal sections, we may choosing a $K$-basis $\{ e_1,\ldots,e_l,f_1,\ldots,f_m\}$ of $H^0(W_{-1})$ such that $\{ e_1,\ldots,e_l\} $ is a $K$-basis for $H^0(W_{-2})$. This then gives us an $\cur{R}_K$-basis $\{ e_1,\ldots,e_l,f_1,\ldots,f_m\}$ of $W_{-1}$ such that $\{ e_1,\ldots,e_l\} $ is an $\cur{R}_K$-basis for $W_{-2}$ and $D(e_i)=D(f_j)=0$ for all $i,j$.

Now, since $W_0/W_{-2}$ is also constant (as it is dual to $D^f_\cur{R}(A^\vee)$ by orthogonality), it follows that we may choose an $\cur{R}_K$-basis $\{ e_1,\ldots,e_l,f_1\ldots,f_m,g_1,\ldots,g_n\}$ for $D_{\cur{R}}(A)$ such that $\nabla(g_k)$ is in the $\cur{R}_K$-span of the $\{e_i\}$ for all $k$, the claim is now that we can modify this basis such that $D(g_k)$ is in fact in the $K$-span of the $\{e_i\}$. But following through the procedure of Proposition 5.2.6 of \cite{kedlayafiniteness} easily does this.
\end{proof}

We can now complete the proof of Theorem \ref{wmc2}, and therefore of Theorem \ref{wmc}.

\begin{proof}[Proof of Theorem \ref{wmc2}] As already remarked, it suffices to show that $\mathrm{WD}(W_\bu) = M_\bu [-1]$.

Letting $\{ e_1,\ldots,e_l,f_1\ldots,f_m,g_1,\ldots,g_n\}$ be as in the previous proposition, it follows from the construction of Marmora that the associated Weil--Deligne representation $\mathrm{WD}(D_{\cur{R}}(A))$ is actually defined over $K$ and is given by the $K$-span of $$\{ e_1,\ldots,e_l,f_1\ldots,f_m,g_1,\ldots,g_n\}$$  inside $D_{\cur{R}}(A)$. The monodromy operator $N$ is then just given by the action of $D$, and therefore takes the form
$$ \left( \begin{matrix} 0 & 0 & A \\ 0 & 0 & 0 \\ 0 & 0 & 0 
\end{matrix} \right)
$$
where $A\in \mathrm{Hom}_K(\langle g_k\rangle,\langle e_i \rangle)$. This visibly satisfies $N^2=0$.

 If $N=0$, then this implies that $D_{\cur{R}}$ admits a basis of horizontal sections, and the monodromy filtration is trivial one, we therefore need to show in this case that the weight filtration is also trivial, that is that $W_{-2}=0$ and $W_{-1}=W_0$. But if $D_{\cur{R}}(A)$ admits a basis of horizontal sections, then Proposition \ref{key1} implies that $D^f_{\cur{R}}(A)=D_{\cur{R}}(A)$, i.e. $W_{-1}=W_0$. By similar considerations on the dual abelian variety $A^\vee$ was can see that we must have $D^f_{\cur{R}}(A^\vee)=D_{\cur{R}}(A^\vee)$ and hence orthogonality implies that $W_{-2}=D^t_{\cur{R}}(A)=0$. 

If $N\neq 0$ then the monodromy filtration is given by $M_{-1}=\mathrm{im}\; N,M_0=\ker N$ and $M_1$ is everything, we therefore have that $M_{-1}\subset \mathrm{WD}(W_{-2})$, $M_0 \supset  \mathrm{WD}(W_{-1})$ and $M_1= \mathrm{WD}(W_0)$. But again, Proposition \ref{key1} implies that we must actually have $M_0 =  \mathrm{WD}(W_{-1})$ (since otherwise we would have a larger constant submodule that $D^f_\cur{R}(A)$) and hence applying this to the dual abelian variety and using orthogonality gives $M_{-1}= \mathrm{WD}(W_{-2})$.
\end{proof}

We can also prove a version of the N\'{e}ron--Ogg--Shafarevich criterion, as follows. Let $A/k\lser{t}$ denote an arbitrary abelian variety, i.e. not necessarily with semistable reduction, and $\cur{A}/k\pow{t}$ its N\'{e}ron model. Then $\cur{A}^\circ_0$, the connected component of the special fibre, is an extension
$$ 0 \rightarrow L_0 \rightarrow \cur{A}_0^\circ \rightarrow B_0
$$
of an abelian variety $B_0$ by a smooth affine algebraic group $L_0$, and $L_0$ is an extension 
$$ 0 \rightarrow T_0 \rightarrow L_0 \rightarrow U_0 \rightarrow 0
$$
of a unipotent group by a torus. $T_0$ is then the maximal torus of $\cur{A}_0^\circ$, and we refer to its dimension $\mu$ as the reductive rank of $\cur{A}_0$. The dimension $\lambda$ of $U_0$ is called the unipotent rank of $\cur{A}_0$ and the dimension $\alpha$ of $B_0$ is called the abelian rank of $\cur{A}_0$. If $n$ is the dimension of $A$ we therefore have $n=\alpha+\mu+\lambda$. Note that good reduction of $A$ is equivalent to having $\mu=\lambda=0$ and semistable reduction of $A$ is equivalent to having $\lambda=0$. These numbers are the same for $A$ and for its dual abelian variety $A'$.

Even without assuming semistable reduction, we still get a filtration
$$ D^t_\cur{R}(A) \subset D^f_\cur{R}(A) \subset D_\cur{R}(A)
$$
where, by Dwork's trick, we have
\begin{align*} D^t_\cur{R}(A) &\cong H^1_\rig(T_0/K)^\vee \otimes_K \rk \\
D^f_\cur{R}(A) &\cong H^1_\rig(\cur{A}_0^\circ/K)^\vee\otimes_K\rk\\
\frac{D^f_\cur{R}(A)}{D^t_\cur{R}(A)} &\cong H^1_\rig(B_0/K)^\vee\otimes_K\rk
\end{align*}
and, as in the proof of Proposition \ref{key1}, $D^f_\cur{R}(A)$ is the largest constant submodule of $D_\cur{R}(A)$. We therefore have 
\begin{align*} \mathrm{rk}_{\rk} D_\cur{R}(A) &=2n \\
\mathrm{rk}_{\rk} D^f_\cur{R}(A) &= \mu + 2\alpha \\
\mathrm{rk}_{\rk} D^t_\cur{R}(A) &=\mu .
\end{align*}
Orthogonality in the form used above no longer holds, however, we do always have
$$ D_\cur{R}^t(A) = D_\cur{R}^f(A) \cap D_\cur{R}^f(A')^\perp
$$
where $\perp$ denotes the orthogonal subspace with respect to the Weil pairing
$$ D_\cur{R}(A) \otimes_{\cur{R}} D_\cur{R}(A') \rightarrow \rk(1)
$$
(see Th\'{e}or\`{e}me 5.2 of \cite{sga719}).

\begin{theo} Let $A/k\lser{t}$ be an abelian variety, not necessarily assumed to have semistable reduction.
\begin{enumerate}
\item $A$ has good reduction iff $D_\cur{R}(A)$ admits a basis of horizontal sections. 
\item $A$ has semistable reduction iff $D_\cur{R}(A)$  admits a unipotent basis, that is a basis $e_1,\ldots,e_n$ such that $\nabla(e_j)\in \cur{R}_Ke_1+ \ldots + \cur{R}_K e_{j-1}$ for all $j$.
\end{enumerate} 
\end{theo}

\begin{proof} \begin{enumerate} \item If $A$ has good reduction, then $\alpha=n$, and hence $D_\cur{R}(A)=D^f_\cur{R}(A)$ has a basis of horizontal sections. Conversely, if $D_\cur{R}(A)$ has a basis of horizontal sections, then $D^f_\cur{R}(A)=D_\cur{R}(A)$ and hence $\mu+2\alpha=2n$. Since $D_\cur{R}(A')$ must also admit a basis of horizontal sections, we also have $D_\cur{R}^f(A')^\perp=0$ and hence $\mu=0$. Therefore $\alpha=n$ and $A$ has good reduction. 
\item If $A$ has semistable reduction, then $\lambda=0$ and hence $n=\alpha+\mu$. Therefore we have $\mathrm{rk}_{\rk}D_\cur{R}^f(A')^\perp=2n-2\alpha-\mu=\mu$, and hence the inclusion
$$ D^t_\cur{R}(A)\subset D^f_\cur{R}(A')^\perp
$$
must be an equality. Therefore both $D^t_\cur{R}(A)$ and $D_\cur{R}/D^t_\cur{R}(A)$ must be constant, and hence $D_\cur{R}(A)$ is unipotent.

Conversely, suppose that $D_\cur{R}(A)$ is unipotent. Then since $A$ admits has semistable reduction over some finite separable extension of $k\lser{t}$, it follows from what we have just proved that there exists a finite \'{e}tale map $\cur{R}_K \rightarrow \cur{R}_{K'}$ (with $K'/K$ some finite unramified extension) such that the base change to $D_\cur{R}(A) \otimes \cur{R}_{K'}$ is unipotent of level $2$. Hence $D_\cur{R}(A)$ itself must be unipotent of level $2$, and by duality the same is true for $D_\cur{R}(A')$. The quotient $D_\cur{R}(A')/D^f_\cur{R}(A')$ must therefore be constant,  hence $D^f_\cur{R}(A')^\perp \subset D^f_\cur{R}(A)$. It follows that $D^t_\cur{R}(A) = D^f_\cur{R}(A')^\perp$ and equating ranks gives
$$  \mu = 2n- 2\alpha-\mu.
$$
Hence $n=\alpha+\mu$ and so $\lambda=0$, i.e. $A$ has semistable reduction.
\end{enumerate}
\end{proof}

As mentioned previously, these results do not actually use our construction of $H^i_\rig(X/\ekd)$ at all, although we expect that it will be necessary to use methods similar to ours, or rather an extension of our methods to a more robust cohomological formalism, in order to construct a weight spectral sequence in general. In the remainder of this section, we will show how one can immediately deduce a version of the weight-monodromy conjecture describing the interaction of the weight and monodromy filtrations on the cohomology of an open (i.e. affine) smooth curve. 

\begin{theo} \label{wmcp} Let $X$ be a smooth curve over $k\lser{t}$, and $i\geq0$. Let $\overline{X}$ be a smooth compactification of $X$, and assume that $D:=\overline{X} \setminus X$ is geometrically reduced. Then there exists a canonical filtration $^gW_\bu$ on $H^i_\rig(X/\cur{R}_K)$, called the geometric weight filtration, such that the graded pieces $\mathrm{Gr}_n^{^gW}$ are quasi-pure of weight $n$. 
\end{theo}

Since we treated the case $D=\emptyset$ in the previous section, we will assume that $D\neq\emptyset$. The case $i=2$ is trivial, and the case $i=0$ is entirely straightforward, since $H^0_\rig(X/\cur{R}_K)$ is just a direct sum of constant modules $\cur{R}_K$. So again the only case of interest is $i=1$.

\begin{lem} There exists an excision exact sequence
$$ 0 \rightarrow H^1_\rig(\overline{X}/\ekd) \rightarrow H^1_\rig(X/\ekd) \rightarrow H^0_\rig(D/\ekd)(-1) \rightarrow H^2_\rig(\overline{X}/\ekd) \rightarrow 0
$$ 
of $(\varphi,\nabla)$-modules over $\ekd$.
\end{lem}

\begin{proof} We know that there exists a corresponding exact sequence
$$ 0 \rightarrow H^1_\rig(\overline{X}/\ek) \rightarrow H^1_\rig(X/\ek) \rightarrow H^0_\rig(D/\ek)(-1) \rightarrow H^2_\rig(\overline{X}/\ek) \rightarrow 0
$$ 
in $\ek$-valued cohomology, which is an exact sequence of $(\varphi,\nabla)$-modules. Hence the claim follows immediately by full faithfulness of the base extension functor $\ekd\rightarrow \ek$, i.e. Theorem 5.1 of \cite{kedfull}.
\end{proof}

\begin{proof}[Proof of Theorem \ref{wmcp}]
Base changing the excision exact sequence to $\cur{R}_K$ gives an exact sequence 
$$ 0 \rightarrow H^1_\rig(\overline{X}/\rk) \rightarrow H^1_\rig(X/\rk) \rightarrow H^0_\rig(D/\rk)(-1) \rightarrow H^2_\rig(\overline{X}/\rk) \rightarrow 0
$$ 
and we define
$$^gW_1 = H^1_\rig(\overline{X}/\cur{R}_K) ,\;\;
^gW_2 = H^1_\rig(X/\cur{R}_K) .
$$
It follows from the results of the previous section that $\mathrm{Gr}^{^gW}_1= H^1_\rig(\overline{X}/\cur{R}_K) $ is quasi-pure of weight $1$, and it is clear that $\mathrm{Gr}^{^gW}_2$ is a sub-module of $H^0_\rig(D/\cur{R}_K)(-1)$, and therefore pure of weight $2$. Theorem \ref{wmcp} is proven.
\end{proof}

\begin{defn} Let $X$ be a curve over $k\lser{t}$, and $i\geq0$. Define
$$ H^i_{c,\rig}(X/\rk):= H^i_{c,\rig}(X/\ekd)\otimes_{\ekd} \rk ,
$$
these are a $\pn$-modules over $\rk$.
\end{defn}

Then Poincar\'{e} duality immediately implies the following.

\begin{cor}  Let $X$ be a smooth curve over $k\lser{t}$, and $i\geq0$. Let $\overline{X}$ be a smooth compactification of $X$, and assume that $D:=\overline{X} \setminus X$ is geometrically reduced. Then there exists a canonical filtration $^gW_\bu$ on $H^i_{c,\rig}(X/\cur{R}_K)$, called the geometric weight filtration, such that the graded pieces $\mathrm{Gr}_n^{^gW}$ are quasi-pure of weight $n$. 
\end{cor}

\subsection{Independence of $\ell$}

In this section we use the results on weight monodromy from the previous section to prove independence of $\ell$-results for the cohomology of varieties over $k\lser{t}$ which include the case $\ell=p$.

\begin{defn} \begin{itemize} 
\item Let $X$ be a smooth proper variety over $k\lser{t}$, so that $H^i_\mathrm{cris}(X/\ekd)$ is defined as a $\pn$-module over $\ekd$. Then we define $$H^i_p(X):=\mathrm{WD}(H^i_\mathrm{cris}(X/\cur{R}_K)),$$ this is a $K^\mathrm{un}$-valued Weil--Delgine representation over $k\lser{t}$.
\item Let $X$ be a smooth curve, so that $H^i_\rig(X/\ekd)$ is defined as a $\pn$-module over $\ekd$. Then we define $$H^i_p(X):=\mathrm{WD}(H^i_\rig(X/\cur{R}_K)),$$ this is a $K^\mathrm{un}$-valued Weil--Delgine representation over $k\lser{t}$.
\item  Let $X$ be a smooth curve, so that $H^i_{c,\rig}(X/\ekd)$ is defined as a $\pn$-module over $\ekd$. Then we define $$H^i_{c,p}(X):=\mathrm{WD}(H^i_{c,\rig}(X/\cur{R}_K)),$$ this is a $K^\mathrm{un}$-valued Weil--Delgine representation over $k\lser{t}$.
\end{itemize}
\end{defn} 

Then following Deligne, these are objects that can be compared with the families of $\ell$-adic representations $\{ H^i_\et(X_{k\lser{t}^\mathrm{sep}},\Q_\ell) \}_{\ell\neq p},\{ H^i_{c,\et}(X_{k\lser{t}^\mathrm{sep}},\Q_\ell) \}_{\ell\neq p}$, or rather with the family of associated Weil--Deligne representations. $\{ H^i_\ell(X)\}_{\ell\neq p},\{ H^i_{c,\ell}(X)\}_{\ell\neq p}$.

\begin{defn}[\cite{delconst}, \S8] \begin{enumerate} \item Let $F/E$ be an extension of characteristic $0$ fields, and $V$ an $F$-valued Weil--Deligne representation of $k\lser{t}$. Then we say that $F$ is defied over $E$ if for any algebraically closed field $\Omega$ containing $F$, $V\otimes \Omega$ is isomorphic to all of its $\mathrm{Aut}(\Omega/E)$-conjugates.
\item Let $\{E_i\}_{i\in I}$ be a family of field extensions of some fixed characteristic $0$ field $E$, and $\{V_i\}_{i\in I}$ a family of $\Omega_i$-valued Weil--Deligne representations of $k\lser{t}$. Then we say that $\{V_i \}_{i\in I}$ is \emph{compatible} if each $V_i$ is defined over $E$, and for any $i,j$ and any algebraically closed field $\Omega$ containing $E_i$ and $E_j$, the Weil--Deligne representations $V_i\otimes\Omega$ and $V_j\otimes\Omega$ are isomorphic.
\end{enumerate}
\end{defn}

Then we have the following, more easily checkable, characterisation of compatibility.

\begin{lem}[\cite{delconst}, \S8] A family $\{V_i\}_{i\in I}$ a family of $\Omega_i$-valued Weil--Deligne representations of $k\lser{t}$ as above is compatible if and only if for all $k$ the character
$$ \mathrm{Tr}(-|\mathrm{Gr}_k^MV_i): W_{k\lser{t}}\rightarrow E_i
$$
of the $k$-th graded piece of the monodromy filtration has values in $E$ and is independent of $i$.
\end{lem}

As with the weight-monodromy conjecture, the main ingredient of the restricted version of $\ell$-indepdence we wish to explain is essentially a result on the $p$-divisible groups of abelian varieties (or more generally, $1$-motives). If $A$ is an abelian variety over $k\lser{t}$, we abuse notation slightly and write $V_\ell(A)$ for the Weil--Deligne representation associated to the $\ell$-adic Tate module of $A$, and $V_p(A)$ for the Weil--Deligne representation associated to $D_\cur{R}(A)$. 

\begin{prop} Let $A$ be an abelian variety over $k\lser{t}$. Then the family of Weil-Delgine representations
$$ \left\{ V_\ell(A) \right\}_\ell
$$
as $\ell$ ranges over all primes (including $\ell=p$) is compatible. 
\end{prop}

\begin{proof} First suppose that $A$ has semistable reduction. In this case we know exactly what the graded piece of the monodromy filtration look like - if we have an exact sequence
$$ 0 \rightarrow T_0 \rightarrow \cur{A}_0^\circ \rightarrow B_0 \rightarrow 0
$$
expressing the connected component of the special fibre of the N\'{e}ron model as an extension of an abelian variety by a torus, then for all $\ell$, including $\ell=p$ we have
$$ \mathrm{Gr}_{-1}^MV_\ell(A) \cong V_\ell(T_0),\;\; \mathrm{Gr}_{0}^MV_\ell(A) \cong V_\ell(B_0),\;\; \mathrm{Gr}_{1}^MV_\ell(A) \cong V_\ell(T'_0)^\vee(1)
$$
where $T'_0$ is the torus occurring in the corresponding exact sequence for the dual abelian variety $A'$. The claim then follows from independence of $\ell$ for tori and abelian varieties over finite fields.

In general, we use the theory of 1-motives from \cite{1mot}, which gives the graded parts of the monodromy filtration $\mathrm{Gr}_k^MV_\ell(A)$ as:
\begin{itemize}
\item the Tate module ($\ell$-adic or $p$-adic) of a torus $T'$ over $k\lser{t}$;
\item the Tate module ($\ell$-adic or $p$-adic) of an abelian variety $A'$ over $k\lser{t}$ with potentially good reduction;
\item the Weil--Deligne representation ($\ell$-adic or $p$-adic) associated to some continuous homomorphism
$$ \rho: \mathrm{Gal}(k\lser{t}^\mathrm{sep}/k\lser{t})\rightarrow \mathrm{GL}_n(\Z)
$$
(necessarily with finite image).
\end{itemize}
Since $\ell$-independence for the third of these is clear, and the first is of the same form but with a Tate twist, we can reduce to the case where $A$ has potentially good reduction, and the monodromy filtration is trivial. 

Then (since $k$ is finite) there exists some finite, separable, totally ramified extension $F/k\lser{t}$ such that $A_F$ has good reduction, let $A'_k$ denote the corresponding abelian variety over $k$. Then $\mathrm{Gal}(F/k\lser{t})$ acts on $A'_k$ via $k$-automorphisms, and exactly as in the $\ell$-adic case, we can describe the $p$-adic Tate module $V_p(A)$ as follows: the monodromy operator $N$ is trivial, and the Galois representation is given by the induced action of $\mathrm{Gal}(F/k\lser{t})$ on $H^1_\mathrm{cris}(A'_k/K)^\vee \otimes_K K^\mathrm{un}$. Hence it suffices to show that for any subgroup $G\subset \mathrm{Aut}_k(A'_k)$, the character of the induced action on $H^1_\mathrm{cris}(A'_k/K)$ is equal to that character of the induced action on $H^1_\et(A'_{\bar{k}},\Q_\ell)$ for all $\ell\neq p$. Since crystalline cohomology of abelian varieties coincides with the Dieudonne module of the associated $p$-divisible group, this is exactly the statement of (for example) the Corollary in Chapter V, \S2 of \cite{pdiv}.
\end{proof}

\begin{cor} Let $X$ be a smooth and proper variety over $k\lser{t}$ of dimension $n$, and $i\in \{0,1,2n-1,2n \}$. Then the family of Weil--Deligne representations
$$ \left\{ H^i_\ell(X) \right\}_\ell
$$
as $\ell$ ranges over all primes (including $\ell=p$) is compatible.
\end{cor}

\begin{proof} As before, the only real case of interest is $i=1$, which follows from the previous proposition.
\end{proof}

\begin{cor} \label{lindepopen}Let $X/k\lser{t}$ be a smooth curve, and $i\geq0$. Then the two families of Weil--Deligne representations
$$ \left\{ H^i_\ell(X) \right\}_\ell, \;\; \left\{ H^i_{c,\ell}(X) \right\}_\ell
$$
as $\ell$ ranges over all primes (including $\ell=p$) are both compatible.
\end{cor}

\begin{proof} It suffices to treat the case of cohomology without supports, since the supported version then follows from Poincar\'{e} duality. Let $\overline{X}$ be a smooth compactification for $X$, and $D=\overline{X} \setminus X$, we may assume that $D\neq \emptyset$. The only case of interest again is $i=1$, and for all $\ell$ we have the excision exact sequence
$$ 0 \rightarrow H^1_\ell(\overline{X}) \rightarrow H^1_\ell (X)  \rightarrow H^0_\ell(D)(-1)  \rightarrow H^2_\ell(\overline{X}) \rightarrow 0.
$$
Them some straightforward linear algebra shows that
\begin{align*} \mathrm{Gr}_{-1}^M H^1_\ell(X) &= \mathrm{Gr}_{-1}^M H^1_\ell(\overline{X}) \\
\mathrm{Gr}_{1}^M H^1_\ell(X) &= \mathrm{Gr}_{1}^M H^1_\ell(\overline{X}) 
\end{align*}
and that we have an exact sequence
$$ 0 \rightarrow \mathrm{Gr}_{0}^M H^1_\ell(\overline{X}) \rightarrow \mathrm{Gr}_{0}^M H^1_\ell(X) \rightarrow W_\ell \rightarrow 0
$$
where 
$$ W_\ell:= \ker\left( H^0_\ell(D,K)(-1)  \rightarrow H^2_\ell(\overline{X}) \right).
$$
Since we know $\ell$-independence for both $ H^0_\ell(D,K)(-1)$ and $H^2_\ell(\overline{X}) $, we therefore know $\ell$-independence for $W_\ell$. Since we know $\ell$-independence for $\mathrm{Gr}_{k}^M H^1_\ell(\overline{X}) $, $\ell$-independence for $\mathrm{Gr}_{k}^M H^1_\ell(X) $ then follows.
\end{proof}

\appendix

\section{Non-embeddable varieties and descent}\label{nonembsec}

In this appendix, for completeness as much as any other reason, we show how to extend the construction of rigid cohomology $H^i_\rig(X/\ekd,\cur{E})$, which applied to embeddable varieties $X/k\lser{t}$, to those which are not necessarily embeddable. The method, using Zariski descent, is completely standard, and we will allow $k$ to be an arbitrary field of characteristic $p$.

\begin{defn} Let $X$ be a $k\lser{t}$-variety. An embedding system for $X$ is a simplicial triple $(X_\bu,Y_\bu,\mathfrak{P}_\bu)$ where $X_\bu\rightarrow X$ is a Zariski hyper-cover of $X$, each $(X_n,Y_n,\mathfrak{P}_n)$ is a smooth and proper frame over $\cur{V}\pow{t}$ and all the face maps $\mathfrak{P}_n\rightarrow \mathfrak{P}_{n-1}$ are smooth around $X_n$. A Frobenius $\varphi$ on an embedding system $(X_\bu,Y_\bu,\mathfrak{P}_\bu)$ is an endomorphism $\varphi:\mathfrak{P}_\bu\rightarrow \mathfrak{P}_\bu$ lifting the absolute $q$-power Frobenius on $P_\bu=\mathfrak{P}_\bu \otimes_{\cur{V}\pow{t}} k\pow{t}$. 
\end{defn}

\begin{prop} Every $k\lser{t}$-variety $X$ admits an embedding system $(X_\bu,Y_\bu,\mathfrak{P}_\bu)$ with a Frobenius lift $\varphi$.
\end{prop}

\begin{proof} Let $\left\{U_i\right\}$ be a finite open affine covering for $X$. Then there exists an embedding $U_i\rightarrow \mathbb{P}^{n_i}_{k\lser{t}}$ for some $n_i$, and we let $\overline{U}_i$ be the closure of $U_i$ in $\mathbb{P}^{n_i}_{k\pow{t}}$. We can thus consider the frame $(U, \overline{U}, \mathscr{U})$ where $U=\coprod_i U_i$, $\overline{U}=\coprod_i\overline{U}_i$ and $\mathfrak{U}=\coprod_i \widehat{\mathbb{P}}^{n_i}_{\mathcal{V}\pow{t}}$. Now define $U_n=U\times_X\ldots\times_X U$, with $n$ copies of $U$, and similarly define  $\overline{U}_n=\overline{U}\times_k\ldots \times_k \overline{U}$ and $\mathfrak{U}_n=\mathfrak{U}\times_{\mathcal{V}\pow{t}}\ldots \times_{\mathcal{V}\pow{t}}\mathfrak{U}$, fibre product in the category of formal $\mathcal{V}$-schemes. Then we have a simplicial triple $(U_\bullet, \overline{U}_\bullet,\mathfrak{U}_\bullet)$, and we get a framing system $(X_\bullet, Y_\bullet,\mathfrak{P}_\bullet)$ for $X$ by taking $X_n=U_n$, $Y_n$ to be the closure of $X_n$ in $\overline{U}_n$ and $\mathfrak{P}_n=\mathfrak{U}_n$. Since each $\mathfrak{P}_n$ is a disjoint union of products of projective space, the simplicial formal scheme $\mathfrak{P}_\bu$ admits a Frobenius lift.
\end{proof}

Let $X$ be $k\lser{t}$-variety, and $(X_\bu,Y_\bu,\mathfrak{P}_\bu)$ an embedding system for $X$ with a Frobenius lift $\varphi$. Then for some $\sh{E}\in(F\text{-})\mathrm{Isoc}^\dagger(X/\ekd)$ or $(F\text{-})\mathrm{Isoc}^\dagger(X/K)$ we can realise each $\sh{E}|_{X_n}$ on $(X_n,Y_n,\mathfrak{P}_n)$ to give a compatible collection of modules with connection on the simplicial rigid space $]Y_\bu[_{\mathfrak{P}_\bu}$.

\begin{defn} Define $H^i_\rig((X_\bu,Y_\bu,\mathfrak{P}_\bu)/\ekd,\sh{E}):= H^i(]Y_\bu[_{\mathfrak{P}_\bu},\sh{E}\otimes \Omega^*_{]Y_\bu[_{\mathfrak{P}_\bu}/S_K})$, these are vector spaces over $\ekd$.
\end{defn}

Of course, we expect that $H^i_\rig((X_\bu,Y_\bu,\mathfrak{P}_\bu)/\ekd,\sh{E})$ only depends on the pair $(X.\sh{E})$, and when $X$ is embeddable coincides with the rigid cohomology defined in previous sections. Happily this is the case.

\begin{lem} \label{zhd1} Suppose that $X$ is embeddable, and $(X_\bu,Y_\bu,\mathfrak{P}_\bu)$ is as above. Then there is a natural isomorphism
$$ H^i_\rig((X_\bu,Y_\bu,\mathfrak{P}_\bu)/\ekd,\sh{E})\cong H^i_\rig(X/\ekd,\sh{E})
$$
which is compatible with Frobenius when $\sh{E}\in F\text{-}\mathrm{Isoc}^\dagger(X/\ekd)$ and the connection when $\sh{E}\in \mathrm{Isoc}^\dagger(X/K)$.
\end{lem}

\begin{proof} Let $(X,Y,\mathfrak{P})$ be a smooth and proper frame containing $X$, we form a new embedding system for $X$ as follows. Embed each connected component of $X_n$ into $Y$ via the canonical map $X_n\rightarrow X\rightarrow Y$, so we get an smooth and proper frame $(X_n,Y'_n:=\coprod_{k_n} Y ,\mathfrak{P}'_n:=\coprod_{k_n} \mathfrak{P})$ for all $n$. These fit together to form an embedding system $(X_\bu,Y'_\bu,\mathfrak{P}'_\bu)$. Then we define $\mathfrak{P}''_n=\mathfrak{P}_n\times_{\cur{V}\pow{t}}\mathfrak{P}'_n$ and $Y''_n$ to be the closure of $X_n$ inside $\mathfrak{P}''_n$ for the diagonal embedding, so that we have a diagram of embedding systems
$$ \xymatrix{ & (X_\bu,Y''_\bu,\mathfrak{P}''_\bu) \ar[dr]^{p_2}\ar[dl]_{p_1}& \\
(X_\bu,Y_\bu,\mathfrak{P}_\bu) & & (X_\bu,Y'_\bu,\mathfrak{P}'_\bu) }
$$
where each map $\mathfrak{P}''_n\rightarrow \mathfrak{P}_n,\mathfrak{P}'_n$ is smooth around $X_n$, and each $Y''_n\rightarrow Y_n,Y'_n$ is proper. It then follows from the proof of Theorem 4.5 of \cite{rclsf1} that
\begin{align*} \mathbf{R}p_{1*} (\sh{E} \otimes \Omega^*_{]Y''_\bu[_{\mathfrak{P}''_\bu}/S_K}) &\cong \sh{E} \otimes \Omega^*_{]Y_\bu[_{\mathfrak{P}_\bu}/S_K} \\
\mathbf{R}p_{2*} (\sh{E} \otimes \Omega^*_{]Y''_\bu[_{\mathfrak{P}''_\bu}/S_K}) &\cong \sh{E} \otimes \Omega^*_{]Y'_\bu[_{\mathfrak{P}'_\bu}/S_K} 
\end{align*} 
and hence that the natural morphisms 
\begin{align*} H^i_\rig((X_\bu,Y_\bu,\mathfrak{P}_\bu)/\ekd,\sh{E}) &\rightarrow H^i_\rig((X_\bu,Y''_\bu,\mathfrak{P}''_\bu)/\ekd,\sh{E}) \\
H^i_\rig((X_\bu,Y'_\bu,\mathfrak{P}'_\bu)/\ekd,\sh{E}) &\rightarrow H^i_\rig((X_\bu,Y''_\bu,\mathfrak{P}''_\bu)/\ekd,\sh{E})
\end{align*}
are isomorphisms. Thus we may replace the embedding system $(X_\bu,Y_\bu,\mathfrak{P}_\bu)$ by the embedding system $(X_\bu,Y'_\bu,\mathfrak{P}'_\bu)$, so that we have an augmentation $(X_\bu,Y_\bu,\mathfrak{P}_\bu)\rightarrow (X,Y,\mathfrak{P})$ such that $Y_n=\coprod_{k(n)}Y$ and $\mathfrak{P}_n=\coprod_{k(n)}\mathfrak{P}$ for integers $k(n)$. The point is that now repeated application of Lemma 4.4 from \cite{rclsf1} gives a resolution
$$ j_X^\dagger E \simeq \mathrm{Tot}(j_{X_\bu}^\dagger E)
$$
for any sheaf $E$ on $]Y[_\mathfrak{P}$, where $\mathrm{Tot}$ denotes the un-normalised cochain complex associated to a cosimplicial sheaf on $]Y[_\mathfrak{P}$. Hence 
$$ H^i_\rig(]Y[_\mathfrak{P},\sh{E}\otimes \Omega_{]Y[_\mathfrak{P}/S_K} ) \cong H^i_\rig(]Y_\bu[_{\mathfrak{P}_\bu},\sh{E}\otimes \Omega_{]Y_\bu[_{\mathfrak{P}_\bu}/S_K} )
$$
as required.

The isomorphism $H^i_\rig((X_\bu,Y_\bu,\mathfrak{P}_\bu)/\ekd,\sh{E})\cong H^i_\rig(X/\ekd,\sh{E})$ is easily checked to be compatible with ground field extensions and functoriality in $X$, and hence is compatible with Frobenius when $\sh{E}\in F\text{-}\mathrm{Isoc}^\dagger(X/\ekd)$. It is also easily seen to be compatible with the connections. \end{proof}

\begin{prop} Let $(X_\bu,Y_\bu,\mathfrak{P}_\bu)$ be an embedding system for $X$, and $\sh{E}\in \mathrm{Isoc}^\dagger(X/\ekd)$. Then the cohomology $H^i_\rig((X_\bu,Y_\bu,\mathfrak{P}_\bu)/\ekd,\sh{E})$ only depends on $X$ and $\sh{E}$ up to canonical isomorphism.
\end{prop}

\begin{proof} Suppose that $(X'_\bu,Y'_\bu,\mathfrak{P}'_\bu)\rightarrow (X_\bu,Y_\bu,\mathfrak{P}_\bu)$ are embedding systems for $X$. We define $X_{n,m}=X_n\times_X X'_m$ and $\mathfrak{P}_{n,m}=\mathfrak{P}_n\times_{\cur{V}\pow{t}}\mathfrak{P}'_m$, so there is a natural embedding $X_{n,m}\rightarrow \mathfrak{P}_{n,m}$. Let $Y_{n,m}$ denote the closure of $X_{n,m}$, so that $(X_{\bu,\bu},Y_{\bu,\bu},\mathfrak{P}_{\bu,\bu})$ forms a double simplicial object is the category of smooth and proper frames over $\cur{V}\pow{t}$. Then for each fixed $n$, $(X_{n,\bu},Y_{n,\bu},\mathfrak{P}_{n,\bu})$ is an embedding system for the embeddable variety $X_n$, and if fact we get a whole augmented simplicial triple $(X_{n,\bu},Y_{n,\bu},\mathfrak{P}_{n,\bu})\rightarrow (X_n,Y_n,\mathfrak{P}_n)$. Examining the proof of Lemma \ref{zhd1} tells us that
$$ \mathbf{R}\pi_{n*}(\sh{E} \otimes \Omega^*_{ ]Y_{n,\bu}[_{\mathfrak{P}_{n,\bu}}/S_K} ) \cong \sh{E} \otimes \Omega^*_{ ]Y_{n}[_{\mathfrak{P}_{n}}/S_K} 
$$
where $\pi_n: ]Y_{n,\bu}[_{\mathfrak{P}_{n,\bu}}\rightarrow ]Y_n[_{\mathfrak{P}_n}$ is the natural morphism. Hence if we let $\pi:]Y_{\bu,\bu}[_{\mathfrak{P}_{\bu,\bu}}\rightarrow ]Y_\bu[_{\mathfrak{P}_\bu}$ denote the natural morphism, we get an isomorphism 
$$ \mathbf{R}\pi_{*}(\sh{E} \otimes \Omega^*_{ ]Y_{\bu,\bu}[_{\mathfrak{P}_{\bu,\bu}}/S_K} ) \cong \sh{E} \otimes \Omega^*_{ ]Y_{\bu}[_{\mathfrak{P}_{\bu}}/S_K} 
$$
and hence the natural morphism

$$ H^i(]Y_{\bu}[_{\mathfrak{P}_{\bu}},\sh{E} \otimes \Omega^*_{ ]Y_{\bu}[_{\mathfrak{P}_{\bu}}/S_K})\rightarrow H^i(]Y_{\bu,\bu}[_{\mathfrak{P}_{\bu,\bu}},\sh{E} \otimes \Omega^*_{ ]Y_{\bu,\bu}[_{\mathfrak{P}_{\bu,\bu}}/S_K})$$
is an isomorphism. Of course, the same is true replacing $(X_\bu,Y_\bu,\mathfrak{P}_\bu)$ by $(X'_\bu,Y'_\bu,\mathfrak{P}'_\bu)$ and so we get a canonical isomorphism 
$$ H^i(]Y_{\bu}[_{\mathfrak{P}_{\bu}},\sh{E} \otimes \Omega^*_{ ]Y_{\bu}[_{\mathfrak{P}_{\bu}}/S_K}) \rightarrow H^i(]Y'_{\bu}[_{\mathfrak{P}'_{\bu}},\sh{E} \otimes \Omega^*_{ ]Y'_{\bu}[_{\mathfrak{P}'_{\bu}}/S_K})
$$
by composing this isomorphism with the inverse of 
$$ H^i(]Y'_{\bu}[_{\mathfrak{P}'_{\bu}},\sh{E} \otimes \Omega^*_{ ]Y'_{\bu}[_{\mathfrak{P}'_{\bu}}/S_K})\rightarrow H^i(]Y_{\bu,\bu}[_{\mathfrak{P}_{\bu,\bu}},\sh{E} \otimes \Omega^*_{ ]Y_{\bu,\bu}[_{\mathfrak{P}_{\bu,\bu}}/S_K}).$$
These are easily checked to compatible with the isomorphisms to the cohomology of any third embedding system $(X''_\bu,Y''_\bu,\mathfrak{P}''_\bu)$.
\end{proof}

Hence we get well-defined and functorial cohomology groups $H^i_\rig(X/\ekd,\sh{E})$ which admit semi-linear Frobenii when $\sh{E}\in F\text{-}\mathrm{Isoc}^\dagger(X/\ekd)$, and connections when $\sh{E}\in \mathrm{Isoc}^\dagger(X/K)$. These structures are compatible when $\sh{E}\in F\text{-}\mathrm{Isoc}^\dagger(X/K)$. Of course, entirely similar considerations apply to constructing $\ekd$-valued rigid cohomology with compact supports for non-embeddable varieties. The only new ingredient needed is (a repeated application of) the following analogue of Lemma 4.4 from \cite{rclsf1}.

\begin{lem} Let $(X,Y,\mathfrak{P}$) be a $\cur{V}\pow{t}$-frame, and $X=U_1\cup U_2$ an open cover of $X$. Then for any sheaf $E$ on $]Y[_\mathfrak{P}$ there is an exact triangle
$$  \mathbf{R}\Gamma_{]U_1\cap U_2[_\frak{P}}E \rightarrow  \mathbf{R}\Gamma_{]U_1[_\frak{P}}E \oplus  \mathbf{R}\Gamma_{]U_2[_\frak{P}}E \rightarrow  \mathbf{R}\Gamma_{]X[_\frak{P}}E \overset{+1}{\rightarrow}
$$
of complexes of sheaves on $]Y[_\mathfrak{P}$.

\end{lem}

\begin{proof} Choose complements $Z$ and $H_j$ for $X$ and $U_j$ in $Y$ respectively, so that $H_1\cup H_2$ is a complement for $U_1 \cap U_2$. Let 
$$ i: ]Z[_\frak{P} \rightarrow ]Y[_\frak{P},\;\;i_j:]H_j[_\frak{P} \rightarrow ]Y[_\frak{P},\;\;i_{12}:]H_1 \cup H_2[_\frak{P} \rightarrow ]Y[_\frak{P}
$$ 
denote the corresponding open immersions. Note that for any sheaf $F$ on $]Y[_\frak{P}$ we have a diagram
$$ \xymatrix{  F \ar[r]\ar[d] & F \oplus F \ar[r] \ar[d] & F \ar[d] \ar[r]^-{+1} &  \\
\mathbf{R}i_{12*}i_{12}^{-1} F \ar[r] & \mathbf{R}i_{1*}i_{1}^{-1} F\oplus \mathbf{R}i_{2*}i_{2}^{-1} F \ar[r] & \mathbf{R}i_{*}i^{-1} F \ar[r]^-{+1} & \\
}
$$
in which both rows are exact (since $]H_1\cup H_2[_\frak{P}= ]H_1[_\frak{P} \cup ] H_2[_\frak{P}$ and $]H_1\cap H_2[_\frak{P}=]Z[_\frak{P}$). Hence we get an exact triangle
$$ \left(  F \rightarrow \mathbf{R}i_{12*}i_{12}^{-1}F\right) [1]  \rightarrow \left( F \rightarrow  \mathbf{R}i_{1*}i_{1}^{-1} F\right) [1]  \oplus \left( F \rightarrow \mathbf{R}i_{2*}i_{2}^{-1}   F\right) [1]  \rightarrow \left(  F \rightarrow \mathbf{R}i_{*}i^{-1} F\right) [1]  \overset{+1}{\rightarrow}.
$$
Now we let $Y'=Y\otimes_{k\pow{t}}k\lser{t}$ and denote by $k:]Y'[_\frak{P}\rightarrow ]Y[_\frak{P}$ the inclusion, applying this to $F=k_*k^{-1}E$ gives the result. \end{proof}

\bibliographystyle{amsplain}\addcontentsline{toc}{section}{References}
\bibliography{/Users/cdl10/Documents/Dropbox/Maths/lib.bib}

\end{document}